\documentclass[11pt]{amsart}

\usepackage[utf8]{inputenc}
\usepackage{mathtools}
\usepackage{amsmath}
\usepackage{amsfonts}
\usepackage{fancyhdr}
\usepackage{amssymb}
\usepackage{amsthm}
\usepackage[margin= 1 in]{geometry}
\usepackage{subcaption}
\usepackage[colorlinks=true, pdfstartview=FitV, linkcolor=blue, citecolor=blue, urlcolor=blue]{hyperref}

\usepackage{xcolor}

\usepackage{tikz}
\usetikzlibrary{cd}
\usepackage{ytableau}
\usepackage{makecell}
\usepackage{bm}

\usepackage[colorinlistoftodos]{todonotes}

\newcommand{\defn}[1]{{\color{darkred}\emph{#1}}} % emphasis of a definition
\definecolor{darkblue}{rgb}{0.0,0,0.7} % darkblue color
\definecolor{darkred}{rgb}{0.7,0,0} % darkred color
\definecolor{darkgreen}{rgb}{0, .6, 0} % darkgreen color
\newcommand{\area}{\mathsf{area}}
\newcommand{\dinv}{\mathsf{dinv}}
\newcommand{\bounce}{\mathsf{bounce}}
\newcommand{\depth}{\mathsf{depth}}
\newcommand{\ddinv}{\mathsf{ddinv}}

\newcommand{\readA}{\mathsf{read}_{A}}
\newcommand{\readD}{\mathsf{read}_{D}}
\newcommand{\coinv}{\mathsf{coinv}}

\newtheorem{theorem}{Theorem}[section]
\newtheorem{lemma}[theorem]{Lemma}
\newtheorem{corollary}[theorem]{Corollary}

\newtheorem{definition}[theorem]{Definition}
\newtheorem{proposition}[theorem]{Proposition}
\newtheorem{example}[theorem]{Example}
\newtheorem{remark}[theorem]{Remark}
\numberwithin{equation}{section}

\title{An area-depth symmetric $q,t$-Catalan polynomial}

\author[J.~Pappe]{Joseph Pappe}
\address[J. Pappe]{Department of Mathematics, UC Davis, One Shields Ave., Davis, CA 95616-8633, U.S.A.}
\email{jhpappe@ucdavis.edu}

\author[D.~Paul]{Digjoy Paul}
\address[D. Paul]{School of Mathematics, Tata Institute of Fundamental Research, 1st Homi Bhaba Road, Colaba, Mumbai -- 400005, India}
\email{dpaul@math.tifr.res.in}

\author[A.~Schilling]{Anne Schilling}
\address[A. Schilling]{Department of Mathematics, UC Davis, One Shields Ave., Davis, CA 95616-8633, U.S.A.}
\email{anne@math.ucdavis.edu}

\date{\today}

\date{\today}

\begin{document}

\begin{abstract}
We define two symmetric $q,t$-Catalan polynomials in terms of the area and depth statistic and in terms of the dinv and dinv of depth
statistics. We prove symmetry using an involution on plane trees. The same involution proves symmetry of the Tutte polynomials. We also provide
a combinatorial proof of a remark by Garsia et al. regarding parking functions and the number of connected graphs on a fixed number of vertices. 
\end{abstract}

\maketitle 

%%%%%%%%%%%%%%%%%%%%%%%%%%%%%%%%%%%%%%%%%%%%%%%%%%%%%%%%%
\section{Introduction}
%%%%%%%%%%%%%%%%%%%%%%%%%%%%%%%%%%%%%%%%%%%%%%%%%%%%%%%%%

The $q,t$-Catalan functions were first introduced in connection with Macdonald polynomials and Garsia--Haiman's theory of diagonal 
harmonics~\cite{GarsiaHaiman.1996} as certain rational functions in $q$ and $t$. They can be obtained as the bigraded Hilbert series of the 
alternating component of a certain module of diagonal harmonics, whose dimension is equal to the Catalan number $C_n = \frac{1}{n+1} \binom{2n}{n}$. 
In terms of symmetric functions, they can be expressed using of the nabla operator and the elementary symmetric functions $e_n$ as
\[
	C_n(q,t) = \langle \nabla e_n, e_n \rangle.
\]
The combinatorics of the $q,t$-Catalan polynomials was developed in various papers~\cite{GarsiaHaglund.2002,Haglund.2003,Haglund.2008}.
In particular, Haglund~\cite{Haglund.2003} gave a combinatorial formula as a sum over all Dyck paths graded by the area and bounce statistics
(see~\eqref{equation.C bounce}). Shortly thereafter, Haiman announced a different combinatorial formula using the area and dinv statistics
(see~\eqref{equation.C dinv}). The zeta map~\cite{AKOP.2002,Haglund.2008} relates these two combinatorial formulas. One of the main open
problems related to the $q,t$-Catalan polynomials $C_n(q,t)$ is a combinatorial proof of its symmetry in $q$ and $t$.

In this paper, we introduce two different $q,t$-analogues of the Catalan numbers $C_n$, which are symmetric in $q$ and $t$. We also get a new formula
for the original $q,t$-Catalan polynomials (see Corollary~\ref{corollary.Cqt depth}).

The first polynomial $F_n(q,t)$ (see~\eqref{equation.F}) is the sum over all Dyck paths graded by area and depth. There are several maps from Dyck 
paths to plane trees, see for example Definitions~\ref{definition.stanley}, \ref{definition.haglund} and \ref{definition.BM} 
from~\cite{Stanley.2015,Haglund.2008,BM.1996} below. The intuition for the depth statistics is that it is the sum over the depths of the various vertices 
in the plane tree. (This is related to a particular labelling of the vertices in a plane tree as given in Definition~\ref{definition.label D}.) The symmetry in 
$q$ and $t$ is proved by defining a duality on plane trees, which switches the area and depth sequence. This duality turns out to be a composition of 
the maps in Definitions~\ref{definition.stanley} and~\ref{definition.haglund}. We prove that on Dyck paths, the corresponding involution is equal to a 
recursively defined involution introduced by Deutsch~\cite{Deutsch.1999}. In particular, this gives an alternative proof of the symmetry of the 
Tutte polynomial for the Catalan matroid~\cite{Ardila.2003}. The polynomials $F_n(q,t)$ satisfy a recursion that relates them to $q,t$-Catalan polynomials 
defined in terms of increasing/decreasing factorizations~\cite[Section 5]{IrvingRattan.2021} and to Hurwitz graphs~\cite{AdinRoichman.2014}.

The second polynomial $G_n(q,t)$ (see~\eqref{equation.G}) is defined in terms of the $\mathsf{dinv}$ and $\mathsf{dinv}$ of depth statistics denoted
$\mathsf{ddinv}$. The $\mathsf{dinv}$ statistics can be formulated using the area sequence, so using the depth sequence instead yields the 
$\mathsf{dinv}$ of depth statistics. This polynomial is also symmetric in $q$ and $t$.

We also address a remark in~\cite{GHQR.2019} stating that the sum of parking functions graded by two to the area is equal to the 
number of connected graphs on a fixed number of vertices.

The paper is organized as follows. In Section~\ref{section.background} we review the definitions associated with the $q,t$-Catalan polynomials
$C_n(q,t)$ and define the polynomials $F_n(q,t)$ and $G_n(q,t)$. In particular, the definition of $\mathsf{depth}$ and $\mathsf{ddinv}$ is given. 
Furthermore, we review several maps from Dyck paths to plane trees. Our main results are stated in Section~\ref{section.results}. In particular, a recursion 
for $F_n(q,t)$ is proved as well as symmetry of $F_n(q,t)$ and $G_n(q,t)$ using an involution $\omega$ that interchanges area and depth. The paper 
concludes with some further results on parking functions.

%%%%%%%%%%%%%%%%%%%%%%%%%%%%%%%%%%%%%%%%%%%%%%%%%%%%%%%%%
\subsection*{Acknowledgments}
We would like to thank Erik Carlsson, Jim Haglund, Marino Romero, and Vasu Tewari for discussions. We thank the AIM Link Homology
Research Community for stimulating talks and discussions.

AS was partially supported by NSF grant DMS--1760329 and DMS--2053350.
JP and AS were partially supported by NSF grant DMS--1764153.

%%%%%%%%%%%%%%%%%%%%%%%%%%%%%%%%%%%%%%%%%%%%
\section{Background and Definitions}
\label{section.background}
%%%%%%%%%%%%%%%%%%%%%%%%%%%%%%%%%%%%%%%%%%%%

In Section~\ref{section.dyck_paths}, we review Dyck paths and their various statistics. In Section~\ref{section.depth}, we define new statistics and 
related polynomials. 
In Section~\ref{section.plane_trees} we give background knowledge on plane trees and their various connections to Dyck paths.
We conclude in Section~\ref{section.trees and parking_functions} with the definition and some results on parking functions and labelled trees.

%%%%%%%%%%%%%%%%%%%%%%%%%%%%%%%%%%%%%%%%%%%%
\subsection{Dyck Paths}
\label{section.dyck_paths}

A \defn{Dyck path} of semilength $n$ is a lattice path with vertices in $\mathbb{Z}_{\geqslant 0} \times \mathbb{Z}_{\geqslant 0} $ from $(0,0)$ to $(n,n)$ 
consisting of North $(1,0)$ and  East $(0,1)$ steps that never passes below the line $y = x$. Let the set of all Dyck paths with semilength $n$ be denoted by 
$D_n$. It is well known that $D_n$ is enumerated by the $n$-th \defn{Catalan number} $C_n = \frac{1}{n+1}\binom{2n}{n}$.

Given $\pi \in D_n$, let the \defn{area sequence} of $\pi$ be the vector $(a_{1}(\pi), a_{2}(\pi), \ldots, a_{n}(\pi))$, where $a_{i}(\pi)$ is the number of 
full unit squares in the $i$-th row completely between $\pi$ and the diagonal $y=x$. Let
\begin{equation}
\label{equation.area}
	\area(\pi) = \sum_{i = 1}^{n} a_{i}(\pi),
\end{equation}
that is, the total number of squares between the path $\pi$ and the diagonal. Note that a Dyck path is uniquely determined by its area sequence. 
Additionally, a vector $(a_1, a_{2}, \ldots, a_{n}) \in \mathbb{Z}_{\geqslant 0}^{n}$ is an area sequence of some Dyck path in $D_{n}$ if and only 
if $a_{1} = 0$ and $0 \leqslant a_{i} \leqslant a_{i-1}+1$ for $2 \leqslant i \leqslant n$. 

Using the area sequence of a Dyck path $\pi$, we can define another statistic on Dyck paths as follows
\begin{equation}
\label{equation.dinv}
	\dinv(\pi) = \lvert \{(i,j) \mid i < j , \, a_i(\pi) = a_j(\pi)\} \cup \{(i,j) \mid i < j , \, a_i(\pi) = a_j(\pi)+1\}\rvert.
\end{equation}
The \defn{$q,t$-Catalan polynomial} is defined as
\begin{equation}
\label{equation.C dinv}
	C_{n}(q,t) = \sum_{\pi \in D_n} q^{\area(\pi)}t^{\dinv(\pi)}.
\end{equation}
The polynomial $C_{n}(q,t)$ is symmetric in $q$ and $t$, that is, $C_n(q,t) = C_n(t,q)$ (see for example~\cite{Haglund.2008}). 
It is an open question to find a combinatorial proof of its symmetry.

To define the bounce statistic of $\pi \in D_n$, we first must construct the bounce path $\mathcal{B}(\pi)$ by the following algorithm: 
\renewcommand\labelenumi{(\arabic{enumi})}
\renewcommand\theenumi\labelenumi
\begin{enumerate} 
	\item Start at the point (0,0).
	\item Continue North until the start of an East step of $\pi$ is met.
	\item Continue East until the diagonal $y = x$ is met.
	\item If the bounce path has reached the point $(n, n)$, then stop. Otherwise go back to step (2).
\end{enumerate}
Let $(0, 0) = (b_{0}, b_{0}), (b_{1}, b_{1}), \ldots, (b_{k}, b_{k}) = (n, n)$ be the points on the diagonal that $\mathcal{B}(\pi)$ touches. Then
\defn{bounce} is defined as
\begin{equation}
\label{equation.bounce}
	\bounce(\pi) = \sum_{i=1}^{k-1} n-b_{i}.
\end{equation}

\begin{proposition} \cite{Haglund.2008}
We have
\begin{equation}
\label{equation.C bounce}
	C_{n}(q, t) = \sum_{\pi \in D_n} q^{\area(\pi)}t^{\bounce(\pi)}.
\end{equation}
\end{proposition}

There exists a bijection $\zeta \colon D_n \to D_n$ on Dyck paths, called the \defn{zeta map}, which has the property that for $\pi \in D_n$
\[
\begin{split}
	\area(\pi) &= \bounce(\zeta(\pi)),\\
	\dinv(\pi) &= \area(\zeta(\pi)).
\end{split}
\]
This proves that~\eqref{equation.C dinv} and~\eqref{equation.C bounce} are equal. The inverse of the zeta map first appeared
connection with nilpotent ideals in certain Borel subalgebras of $\mathfrak{sl}(n)$~\cite{AKOP.2002}. For its connections with the combinatorics of
$q,t$-Catalan polynomials, see~\cite{Haglund.2008}. The zeta map was further studied and generalized 
in~\cite{ALW.2015,CDH.2016,TW.2018,CFM.2020}. For the definition of the zeta map, see~\cite[Theorem 3.15]{Haglund.2008}.
In Proposition~\ref{proposition.zeta} below, we state another formulation of the zeta map in terms of plane trees (which can also serve as the definition).

%%%%%%%%%%%%%%%%%%%%%%%%%%%%%%%%%%%%%%%%%%%%
\subsection{Depth polynomials}
\label{section.depth}

Let $\pi \in D_n$. We produce a labelling for $\pi$ column-by-column using the following algorithm:
\renewcommand\labelenumi{(\arabic{enumi})}
\renewcommand\theenumi\labelenumi
\begin{enumerate} 
	\item In the leftmost column, label all cells directly to the right of a North step with a $0$. 
	\item In the $i$-th column from the left, locate the bottommost cell $c$ in the column that is directly right of a North step; note that such a cell may not exist.
	 From $c$ travel Southwest diagonally until a cell $c'$ that is already labelled is reached. Let $\ell$ be the labelling of $c'$. Label all cells directly to the 
	 right of a North step in the $i$-th column with an $\ell +1$.
\end{enumerate}
Define this to be the \defn{depth labelling} of $\pi$. The \defn{depth sequence} $(d_1(\pi), d_2(\pi), \ldots, d_n(\pi))$ of $\pi$ can be obtained by reading the 
entries of the depth labelling of $\pi$ in the following manner:
\renewcommand\labelenumi{(\arabic{enumi})}
\renewcommand\theenumi\labelenumi
\begin{enumerate} 
	\item Let $v$ be the empty vector. Let $c$ be the cell directly right of the first North step of $\pi$.
	\item Append the label of $c$ to the end of $v$. If the length of $v$ is $n$, then stop and let 
	\[
	    (d_1(\pi), d_2(\pi), \ldots, d_n(\pi)) = v.
	\]
	\item Otherwise, travel Northeast diagonally from $c$ until a cell that is labelled is reached. If this cell exists and has not been seen before, then 
	redefine $c$ to be this cell. If no such cell exists or the cell was already visited before by the algorithm, then consider the set of all cells that have 
	been visited already but have a labelled cell directly above them that has not been visited. Out of this set choose the rightmost one 
	and let $c$ be the cell directly above this cell. Go back to step (2).
\end{enumerate}

\begin{remark}
Note that in the above definition, the rightmost cell of all visited cells with a labelled cell directly above is also the cell in this set with the largest
label. Namely, look at the lowest cell in the same column as $c$, which is labelled. All cells that were already visited but have a labelled cell directly above
them are to the left of this cell on the same diagonal or lower. By the construction of the labels, these cells all have strictly smaller labels.
\end{remark}

Define the \defn{depth} statistic as follows
\begin{equation}
\label{equation.depth}
	\depth(\pi) = \sum_{i=1}^{n} d_{i}(\pi).
\end{equation}
Similar to how $\dinv$ was defined in terms of the area sequence in~\eqref{equation.dinv}, we can associate a ``$\dinv$" type statistic called $\ddinv$ 
to the depth sequence of a Dyck path. Formally, 
\begin{equation}
\label{equation.ddinv}
	\ddinv(\pi) = \lvert \{(i,j) \mid i < j , \, d_i(\pi) = d_j(\pi)\} \cup \{(i,j) \mid i < j , \, d_i(\pi) = d_j(\pi)+1\}\rvert.
\end{equation}

\begin{figure}[t]
	\centering
	\begin{tikzpicture}[scale = .5]
	
	\draw[step=1.0, gray!60, thin] (0,0) grid (9,9);
	
	\draw[gray!60, thin] (0,0) -- (9,9);
	
	\draw[blue!60, line width=2pt] (0,0) -- (0,1) -- (0,2) -- (0,3) -- (1,3) -- (2,3) -- (2,4) -- (3,4) -- (3,5) -- (3,6) -- (4,6) -- (5,6) -- (6,6) -- (6,7) -- (6,8) -- (7,8) -- (7,9) -- (8,9) -- (9,9);
	
	\draw (0.5,0.5) node {$0$};
	\draw (0.5,1.5) node {$0$};
	\draw (0.5,2.5) node {$0$};
	\draw (2.5,3.5) node {$1$};
	\draw (3.5,4.5) node {$2$};
	\draw (3.5,5.5) node {$2$};
	\draw (6.5,6.5) node {$1$};
	\draw (6.5,7.5) node {$1$};
	\draw (7.5,8.5) node {$2$};

	\end{tikzpicture}
	\caption{Example of a Dyck path $\pi \in D_9$ with its depth labelling.}
	\label{figure.depth}
\end{figure}
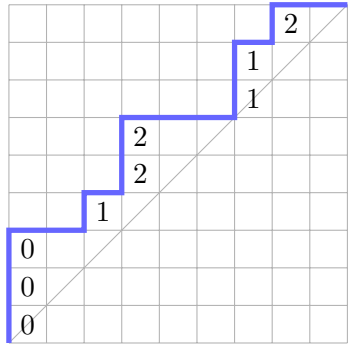

\begin{example}
\label{example.depth}
In Figure~\ref{figure.depth}, a Dyck path $\pi \in D_9$ with its depth labelling is shown. The depth sequence is
$(0,1,1,2,0,1,2,2,0)$. Hence the depth is $\depth(\pi)=9$. Finally
\[
	\{ (1,5), (1,9), (5,9), (2,3), (2,6), (3,6), (4,7), (4,8), (7,8), (2,5), (2,9), (3,5), (3,9), (6,9), (4,6)\}
\]
are pairs contributing to the $\ddinv$ statistic in~\eqref{equation.ddinv}, hence $\ddinv(\pi)=15$.
\end{example}

Next we define two $q,t$-Catalan polynomials using the just introduced statistics:
\begin{equation}
\label{equation.F}
	F_{n}(q, t) = \sum_{\pi \in D_n} q^{\area(\pi)}t^{\depth(\pi)}
\end{equation}
and
\begin{equation}
\label{equation.G}
G_{n}(q,t) = \sum_{\pi \in D_n} q^{\dinv(\pi)}t^{\ddinv(\pi)}.
\end{equation}
We will prove various properties of these polynomials in Section~\ref{section.results}, including that they are symmetric in $q$ and $t$.

\begin{example}
We list the first few polynomials:
\[
\resizebox{0.95\hsize}{!}{$
\begin{array}{|c|c|c|c|}
\hline
n & C_n(q,t) & F_n(q,t) & G_n(q,t)\\[1mm] \hline 
1 & 1 & 1 & 1 \\[1mm] \hline
2 & q+t & q+t &q+t \\[1mm] \hline
3 & q^3 + q^2 t + q t^2 + t^3 + qt & q^3 + q^2 t + q t^2 + t^3 + q t& q^2t^2 + q^3 + t^3 + 2qt\\[1mm] \hline
4 & q^6 + q^5 t + q^4 t^2 + q^3 t^3 + q^2 t^4 + q t^5 + t^6 & q^6 + q^5 t + q^4 t^2 + 2q^3 t^3 + q^2 t^4 + q t^5 + t^6 & q^5 t^2 + q^4t^3 + q^3t^4 + q^2t^5 \\
   &+ q^4 t + q^3 t^2 + q^2 t^3 + q t^4  & + q^4 t + q t^4 & + q^6 + q^4t^2 + q^2t^4 + t^6\\ 
   & + q^3 t + q^2 t^2 + q t^3 & + q^3 t + 2 q^2 t^2 + q t^3& + 2q^3t + 2qt^3 + q^2t + qt^2\\[1mm] \hline
\end{array}
$}
\]
\end{example}

\begin{remark}
Note that $C_n(1,1)=F_n(1,1)=G_n(1,1)=C_n$ are all equal to the $n$-th Catalan number. The difference $F_n(q,t) - C_n(q,t)$ can be written as
$(1-t)(1-q)M_n(q,t)$. Evaluating $M_n(1,1)$ yields the sequence $0,0,0,1,14,124,888,5615,32714,\ldots$, which curiously is the 5-th number
after each 1 in the Riordan array, see~\cite{Sloane}.
Both $(G_n - C_n)/((q-1)(t-1))$ and $(G_n - F_n)/((q-1)(t-1))$ are also conjectured to have positive coefficients. At $q=t=1$, the corresponding sequences
are $0, 0, 0, 1, 11, 83, 530, 3071, 16997, 86778, 436084, \ldots$ and $0, 0, 0, 1, 10, 69, 406, 2183, 11082, 54064,$ $256204, \ldots$, which do not
seem to appear in~\cite{Sloane}.
\end{remark}

%%%%%%%%%%%%%%%%%%%%%%%%%%%%%%%%%%%%%%%%%%%%
\subsection{Plane Trees}
\label{section.plane_trees}

In this paper, all rooted trees are drawn with the root on top and its descendants below. The principal subtrees of a rooted tree $T$ are the rooted 
trees obtained by removing the root of $T$ and considering the children of the root of $T$ to be the new roots of their respective trees. 

\begin{definition}
A \defn{plane tree} is a rooted tree, which either consists only of the root vertex $r$ or it consists recursively of the root $r$ and its principal subtrees
$(T_1,\ldots,T_k)$ which themselves are plane trees. Note that the subtrees are linearly ordered.
Let the set of all plane trees on $n+1$ vertices be denoted by $\mathcal{T}_{n+1}$. 
\end{definition}

Note that $\mathcal{T}_{n+1}$ is also enumerated by the $n$-th Catalan number $C_{n}$. This can be shown by a bijection between $D_{n}$ and 
$\mathcal{T}_{n+1}$. Here, we discuss three such bijections that will be useful to us. The first bijection can, for example, be found 
in~\cite[Page 10]{Stanley.2015}.

\begin{definition}
\label{definition.stanley}
Let the \defn{Stanley map} $\sigma\colon D_{n} \to \mathcal{T}_{n+1}$ be defined as follows:
\renewcommand\labelenumi{(\arabic{enumi})}
\renewcommand\theenumi\labelenumi
	\begin{enumerate} 
		\item Consider the Dyck path $\pi$ as a string $\pi_{1} \pi_{2} \ldots \pi_{2n}$ of length $2n$ in the alphabet \{N, E\} corresponding to the 
		North and East steps of $\pi$.
		\item Start at the root node. Label this as vertex $v$.
		\item For $1\leqslant i \leqslant 2n$, if $\pi_{i} = N$ then add a child to the right of all preexisting children of $v$. Label this new child as $v$. 
		If $\pi_{i} = E$, set $v$ to be the parent of $v$.
	\end{enumerate}
\end{definition}

\begin{example}
The Dyck path of Figure~\ref{figure.depth} corresponds to the plane tree in Figure~\ref{figure.plane tree S} under $\sigma$.
\end{example}

The next bijection is the restriction of a bijection between parking functions and labelled trees to Dyck paths. The bijection on parking functions
can, for example, be found in~\cite{HL.2005} and~\cite[Chapter 5]{Haglund.2008}.

\begin{definition}
\label{definition.haglund}
Let the \defn{Haglund--Loehr map} $\eta\colon D_{n} \to \mathcal{T}_{n+1}$ be defined as follows:
\renewcommand\labelenumi{(\arabic{enumi})}
\renewcommand\theenumi\labelenumi
	\begin{enumerate} 
		\item For each cell in the first column that lies directly right of a North step attach a child to the root vertex. Associate the rightmost child 
		to the topmost cell in the first column, the second rightmost child to the second topmost cell in the first column, and so on such that the 
		leftmost child is associated with the bottommost cell in the first column.
		\item To determine the children of any other vertex $v$, travel on the Northeast diagonal from its associated cell under $\pi$ until it reaches a 
		cell directly to the right of a North step. If this cell exists and is the bottommost cell in its column that is directly right of a North step, then attach $k$ 
		children to $v$, where $k$ is the number of cells in this column that lie directly right of a North step. For each of these new vertices, 
		associate them to the appropriate cell as laid out above.
	\end{enumerate}
\end{definition}

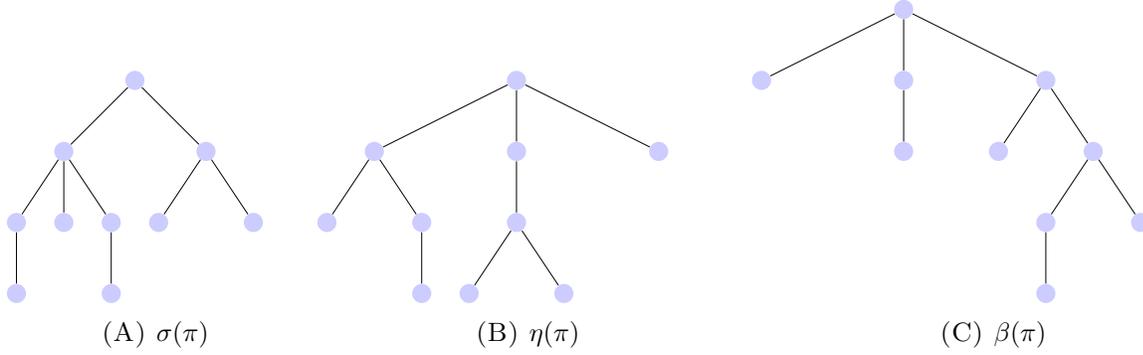
\begin{figure}[t]
\begin{subfigure}[b]{.25\linewidth}
\scalebox{0.7}{
        \centering
        \begin{tikzpicture}
             [scale=.9,auto=center, every node/.style={circle,fill=blue!20}]
    
  \node (a1) at (0,0) {};  
  \node (a2) at (-1.5,-1.5)  {};  
  \node (a3) at (1.5,-1.5) {};  
  \node (b1) at (-2.5,-3) {};
  \node (b2) at (-1.5,-3) {};
  \node (b3) at (-0.5,-3) {};
  \node (b4) at (0.5,-3) {};
  \node (b5) at (2.5,-3) {};
  \node (c1) at (-2.5,-4.5) {};
  \node (c2) at (-0.5,-4.5) {};
  
  \draw (a1) -- (a2); % these are the straight lines from one vertex to another  
  \draw (a1) -- (a3);  
  \draw (a2) -- (b1);
  \draw (a2) -- (b2);
  \draw (a2) -- (b3);
  \draw (a3) -- (b4);
  \draw (a3) -- (b5);
  \draw (b1) -- (c1);
  \draw (b3) -- (c2);
        \end{tikzpicture}
  }      
        \caption{$\sigma(\pi)$}
        \label{figure.plane tree S}
      
    \end{subfigure}%
    \begin{subfigure}[b]{.35\linewidth}
    \scalebox{0.7}{
        \centering
        \begin{tikzpicture}  
    [scale=.9,auto=center, every node/.style={circle,fill=blue!20}]
    
  \node (a1) at (0,0) {};  
  \node (a2) at (-3,-1.5)  {};  
  \node (a3) at (0,-1.5) {};
  \node (a4) at (3,-1.5) {};  
  \node (b1) at (-4,-3) {};
  \node (b2) at (-2,-3) {};
  \node (b3) at (0,-3) {};
  \node (c1) at (-2,-4.5) {};
  \node (c2) at (-1,-4.5) {};
  \node (c3) at (1,-4.5) {};
  
  \draw (a1) -- (a2); % these are the straight lines from one vertex to another  
  \draw (a1) -- (a3);  
  \draw (a1) -- (a4);  
  \draw (a2) -- (b1);  
  \draw (a2) -- (b2); 
  \draw (a3) -- (b3); 
  \draw (b2) -- (c1); 
  \draw (b3) -- (c2);
  \draw (b3) -- (c3);
\end{tikzpicture}
}
        \caption{$\eta(\pi)$}
        \label{figure.plane tree H}

    \end{subfigure}%
    \begin{subfigure}[b]{.4\linewidth}
    \scalebox{0.7}{
        \centering
        \begin{tikzpicture}  
  [scale=.9,auto=center, every node/.style={circle,fill=blue!20}]
    
  \node (a1) at (0,0) {};  
  \node (a2) at (-3,-1.5)  {};  
  \node (a3) at (0,-1.5) {};
  \node (a4) at (3,-1.5) {};  
  \node (b1) at (0,-3) {};
  \node (b2) at (2,-3) {};
  \node (b3) at (4,-3) {};
  \node (c1) at (3,-4.5) {};
  \node (c2) at (5,-4.5) {};
  \node (d1) at (3,-6) {};
  
  \draw (a1) -- (a2); % these are the straight lines from one vertex to another  
  \draw (a1) -- (a3);  
  \draw (a1) -- (a4);  
  \draw (a3) -- (b1);  
  \draw (a4) -- (b2); 
  \draw (a4) -- (b3); 
  \draw (b3) -- (c1); 
  \draw (b3) -- (c2);
  \draw (c1) -- (d1);
\end{tikzpicture}
}
        \caption{$\beta(\pi)$}
        \label{figure.plane tree BM}
\end{subfigure}
    \caption{Plane trees corresponding to the Dyck path $\pi$ of Figure \ref{figure.depth} under $\sigma,\eta$, and $\beta$, respectively.
    \label{figure.plane tree}}
 \end{figure}

\begin{example}
The Dyck path of Figure~\ref{figure.depth} corresponds to the plane tree in Figure~\ref{figure.plane tree H} under $\eta$.
\end{example}

The last map we mention can be found in~\cite{BM.1996}.

\begin{definition}
\label{definition.BM}
Let the \defn{Benchekroun--Moszkowski map} $\beta\colon D_{n} \to \mathcal{T}_{n+1}$ be defined as follows:
\renewcommand\labelenumi{(\arabic{enumi})}
\renewcommand\theenumi\labelenumi
	\begin{enumerate} 
		\item Consider the Dyck path $\pi$ as a string $\pi_{1} \pi_{2} \ldots \pi_{2n}$ of length $2n$ in the alphabet \{N, E\} corresponding to
		the North and East steps of $\pi$. Append $\pi_{0} = E$ to the front of the string.
		\item For each vertex, we attach one of two states: ``Checked" or ``Not Checked''. Start with just the root vertex in the ``Not Checked'' state. 
		\item Recursively consider $\pi_i$ for $i=0,1,\ldots,2n$.
		If $\pi_{i} = E$, then find the set of all closest vertices to the root in the ``Not Checked'' state. Out of these vertices choose the leftmost vertex 
		and label this vertex as $v$. Let $k$ be the number of consecutive North steps directly following $\pi_{i}$. Append $k$ children to $v$ all in 
		``Not Checked'' state. Change the state of vertex $v$ to ``Checked''. If $\pi_{i} = N$, then perform no action on the graph.
	\end{enumerate}
\end{definition}

\begin{example}
The Dyck path of Figure~\ref{figure.depth} corresponds to the plane tree in Figure~\ref{figure.plane tree BM} under $\beta$.
\end{example}

It turns out that $\sigma$ and $\beta$ can be used to obtain the zeta map.

\begin{proposition} \cite{BM.1996}
\label{proposition.zeta}
Let $\pi \in D_n$. Then $\zeta(\pi) = \beta^{-1} \circ \sigma (\pi)$.
\end{proposition}

%%%%%%%%%%%%%%%%%%%%%%%%%%%%%%%%%%%%%%%%%%%%
\subsection{Labelled trees and parking functions}
\label{section.trees and parking_functions}

A \defn{labelled tree} on $n$ vertices is a tree $T$ with vertex set $\{0, 1, \ldots, n-1\}$, where the vertex labelled $0$ is considered to be the root 
of the tree. We also use the convention that in a drawing of a labelled tree any vertex $v$ sits above its children and the labels of its children increase 
from left to right.  Let $\mathcal{L}_{n}$ represent the set of all labelled trees on $n$ vertices. The cardinality of $\mathcal{L}_{n}$ is known to be $n^{n-2}$.

A \defn{coinversion} of $T \in \mathcal{L}_{n}$  is an ordered pair $(i, j)$ such that $j$ is a descendant of $i$ and $0 < i < j$. Denote the number of 
coinversions of a labelled tree $T$ by $\coinv (T)$. Gessel and Wang~\cite{GesselWang.1979}  proved combinatorially that
 \begin{equation} 
 \label{equation: coinv_graph}
	\sum_{G \in \mathcal{C}_n} q^{e(G)} = q^{n-1} \sum_{T \in \mathcal{L}_{n}} (1+q)^{\coinv(T)},
\end{equation}
where $\mathcal{C}_{n}$ is the set of all \defn{labelled connected graphs} with vertex set $\{0, \ldots, n-1\}$ and $e(G)$ is the number of edges in $G$. 
Evaluating at $q = 1$ gives the surprising result that 
\begin{equation}
	\lvert \mathcal{C}_{n} \rvert = \sum_{T \in \mathcal{L}_{n}} 2^{\coinv(T)}.
\end{equation}

\begin{remark}
Note that Gessel and Wang~\cite{GesselWang.1979} studied tree inversions instead of coinversions, but these statistics can be seen to be  jointly 
equidistributed on labelled trees by relabelling vertex $i$ by $n+1-i$ for $i \not = 0$ which was observed by Irving and
Rattan~\cite{IrvingRattan.2021}.
\end{remark}

A \defn{parking function} $P$ on $n$ cars is equivalent to a Dyck path $\pi \in D_n$, where the numbers $1$ through $n$ are placed directly right 
of the North steps of $\pi$ such that each number appears exactly once and the numbers in each column are strictly decreasing. We refer to the 
labels $\{1, \ldots, n\}$ in the parking function $P$ as cars. Denote the set of all parking functions on $n$ cars by $\mathcal{P}_{n}$. The area of 
parking function $P$ is taken to be the area of its corresponding Dyck path. The cardinality of $\mathcal{P}_{n}$ is known to be $(n+1)^{n-1}$
which implies that there exists a bijection with labelled trees on $n+1$ vertices. We review the bijection discovered in~\cite{HL.2005}.

\begin{definition}
\label{def: pf_tree_bijection}
Let the \defn{Haglund--Loehr map} $\lambda\colon \mathcal{P}_{n} \to \mathcal{L}_{n+1}$ be defined as follows:
\renewcommand\labelenumi{(\arabic{enumi})}
\renewcommand\theenumi\labelenumi
	\begin{enumerate} 
		\item Start with the root vertex labelled $0$. For each car labelled $i$ in the first column of the parking function, attach a vertex 
		labelled $i$ to the root $0$.
		\item To determine the children of any other vertex $v$, travel Northeast from its associated car until it reaches another car. If this car 
		exists and is the bottommost car in its column, attach a child labelled $i$ to $v$ for each car $i$ in the column. 
	\end{enumerate}
\end{definition}

Observe that restricting $\lambda$ to the parking functions containing car $i$ in row $i$ recovers $\eta$ of Definition~\ref{definition.haglund}
(by ordering siblings in increasing order and then disregarding the labels on the tree).

Haglund and Loehr~\cite{HL.2005} also defined a function $\tilde{d}_{i}(T)$ on the vertices $0 \leqslant i \leqslant n$ for $T\in \mathcal{L}_{n+1}$ such that 
$\tilde{d}_{0}(T) = 0$ and $\tilde{d}_{j}(T) = \tilde{d}_{i}(T)+k-1$, where vertex $j$ is the $k$-th smallest/leftmost child of vertex $i$. For any 
$P \in \mathcal{P}_{n}$, we have
\begin{equation}
\label{equation: area_tree}
	\area(P) = \sum_{i=0}^{n} \tilde{d}_{i}(\lambda(P)).
\end{equation}

%%%%%%%%%%%%%%%%%%%%%%%%%%%%%%%%%%%%%%%%%%%%
\section{Results}
\label{section.results}
%%%%%%%%%%%%%%%%%%%%%%%%%%%%%%%%%%%%%%%%%%%%

In Section~\ref{section.recursion}, we prove a recursion for the polynomials $F_n(q,t)$.
In Section~\ref{section.dual}, we introduce the notation of a dual plane tree using various reading words.
We use this to prove in Section~\ref{section.symmetry} that $F_n(q,t)$ and $G_n(q,t)$ are symmetric in $q$ and $t$. 
This also gives an expression of the usual Catalan polynomials in terms of the depth and dinv of depth statistics.
In Section~\ref{section.deutsch}, we relate the involution that interchanges depth and area used to prove the symmetry
in Section~\ref{section.symmetry} to an involution by Deutsch~\cite{Deutsch.1999}; this yields an easy proof
of the symmetry of the Tutte polynomials of the Catalan matroid~\cite{Ardila.2003}.
In Section~\ref{section.parking_functions}, we consider the setup of parking functions and
address a remark in~\cite{GHQR.2019}.

%%%%%%%%%%%%%%%%%%%%%%%%%%%%%%%%%%%%%%%%%%%%
\subsection{Recursion for $F_{n}(q,t)$}
\label{section.recursion}

We begin by giving a recursion for $F_n(q,t)$.

\begin{proposition} 
\label{prop: recurrence}
We have $F_0(q,t)=1$ and for any $n \geqslant 1$
\[
	F_{n}(q,t) =  \sum_{k=1}^{n} q^{k-1}t^{n-k} F_{k-1}(q,t) F_{n-k}(q,t).
\]
\end{proposition}

\begin{proof}
Let
\[
	D_{n}(k) = \{ \pi \in D_{n} \mid \pi \text{ first touches the diagonal at } (k, k)\}.
\]
Let $f \colon D_{n}(k) \to D_{k-1} \times D_{n-k}$ be the classical bijection sending
\[
	\pi = \pi_{1} \pi_{2} \ldots \pi_{2n} \quad \mapsto \quad (\pi_{2} \ldots \pi_{2k-1}, \pi_{2k+1} \ldots \pi_{2n}).
\]
Let $f_{1}(\pi)$ and $f_{2}(\pi)$ be the first and second component of $f(\pi)$, respectively. Note that appending a North step to the 
beginning and an East step at the end of a Dyck path of semilength $m$ increases the area by $m$. As $\pi$ is obtained by concatenating 
$N$, $f_{1}(\pi)$, $E$, and $f_{2}(\pi)$, we have  $q^{\area(\pi)} = q^{k-1} q^{\area(f_{1}(\pi))}q^{\area(f_{2}(\pi))}$. Now consider the 
depth labelling of $\pi$. Observe that the labellings of all North steps after $\pi_{2k+1}$ can be uniquely determined by the labelling to 
the right of $\pi_{2k+1}$. Since the labelling to the right of the first North step is $0$ and $(k,k)$ is the first time $\pi$ touches the diagonal, 
we have that the labelling to the right $\pi_{2k+1}$ is $1$. However, looking at the corresponding depth labelling in $f_{2}(\pi)$, this value 
is a zero. Thus, to get from the depth labelling of $f_{2}(\pi)$ to the that of $\pi_{2k+1} \ldots \pi_{2n}$ in $\pi$, we must add $1$ to each 
of the $n-k$ labels. Additionally, from the definition of the depth labelling, we see that the portion of $\pi$ from $(0,1)$ to $(k-1, k)$ corresponding 
to $f_{1}(\pi)$ has the same depth labelling as $f_{1}(\pi)$. This gives us that $t^{\depth(\pi)} = t^{n-k} t^{\depth(f_{1}(\pi))}t^{\depth(f_{2}(\pi))}$. 
Therefore,
\begin{equation}
	\sum_{\pi \in D_{n}(k)} q^{\area(\pi)} t^{\depth(\pi)} = q^{k-1} t^{n-k} F_{k-1}(q, t) F_{n-k}(q, t).
\end{equation}
Summing over $k$ from $1$ to $n$ gives the desired result.
\end{proof}

The recursion in Proposition~\ref{prop: recurrence} relates the polynomials $F_n(q,t)$ to the $q,t$-Catalan polynomials in~\cite[Section 5]{IrvingRattan.2021} 
in terms of increasing/decreasing factorizations and to Hurwitz graphs~\cite{AdinRoichman.2014} since they satisfy the same recurrence. Note that
in~\cite{AdinRoichman.2014} the authors defined a statistics $\mathsf{bmaj}$ on Dyck paths, which corresponds to our depth statistics. However, 
$\mathsf{depth}$ and $\mathsf{bmaj}$ are defined in different ways. In particular, the depth sequence is a refinement of depth, which will be used in subsequent 
sections to define a duality. 

%%%%%%%%%%%%%%%%%%%%%%%%%%%%%%%%%%%%%%%%%%%%
\subsection{Dual plane trees}
\label{section.dual}

We define two labellings of plane trees and an associated reading word to each labelling. 

\begin{definition}
The \defn{labelling $A$} of a plane tree $T$, denoted by $T_{A}$, is defined recursively by the following algorithm:
\renewcommand\labelenumi{(\arabic{enumi})}
\renewcommand\theenumi\labelenumi
	\begin{enumerate}
		\item Label the root as $0$.
		\item For any other vertex $v$, let $m$ be the labelling of its parent $w$. Label $v$ as $m+k-1$, where $v$ is the $k$-th leftmost child of $w$.
	\end{enumerate}
\end{definition}

\begin{definition}
\label{definition.reading A}
Let $T$ be a plane tree with $n+1$ vertices. The \defn{reading word of $T_{A}$}, denoted by $\readA(T)$, is given by the following algorithm:
\renewcommand\labelenumi{(\arabic{enumi})}
\renewcommand\theenumi\labelenumi
	\begin{enumerate}
		\item Start by setting $\readA(T)$ to be an empty vector. Append the labels of the children of the root in increasing order.
		\item If the length of $\readA(T)$ equals $n$, then output $\readA(T)$. Otherwise, consider the set of vertices whose labels have already 
		been added to $\readA(T)$ but whose children's labels have not been added. Find the vertex in this set with the largest label and at least 
		one child. Call this vertex $v$. Append the labels of all the children of $v$ in increasing order.
	\end{enumerate}
\end{definition}

Note that the definition of the reading word in Definition~\ref{definition.reading A} is well-defined. To show this, it suffices to explain why no two 
vertices with the same label will be considered by the definition at the same step. Let $v$ and $w$ be any two vertices that have the same label. 
If one is an ancestor of the other, then they would not be considered at the same point anywhere in the algorithm. Otherwise, consider the closest 
common ancestor of $v$ and $w$ and label it $x$. Let $v'$ (resp. $w'$) be the child of $x$ on the path from $v$ (resp. $w$) to $x$. As the label 
of $w'$ is strictly larger than that of $v'$, $w$ will be considered before $v'$ and thus before $v$ in the algorithm.

\begin{example}
The labelling $T_A$ of the tree $T$ in Figure~\ref{figure.plane tree S} is given in Figure~\ref{figure.TA}. The corresponding reading word is
$\readA(T) = (0,1,1,2,0,1,2,2,0)$.
\end{example}

\begin{figure}[t]
\begin{subfigure}[b]{.25\linewidth}
\scalebox{0.7}{
\begin{tikzpicture}  
  [scale=.9,auto=center, every node/.style={circle,fill=blue!20}]
    
  \node (a1) at (0,0) {$0$};  
  \node (a2) at (-1.5,-1.5)  {$0$};  
  \node (a3) at (1.5,-1.5) {$1$};  
  \node (b1) at (-2.5,-3) {$0$};
  \node (b2) at (-1.5,-3) {$1$};
  \node (b3) at (-0.5,-3) {$2$};
  \node (b4) at (0.5,-3) {$1$};
  \node (b5) at (2.5,-3) {$2$};
  \node (c1) at (-2.5,-4.5) {$0$};
  \node (c2) at (-0.5,-4.5) {$2$};
  
  \draw (a1) -- (a2); % these are the straight lines from one vertex to another  
  \draw (a1) -- (a3);  
  \draw (a2) -- (b1);
  \draw (a2) -- (b2);
  \draw (a2) -- (b3);
  \draw (a3) -- (b4);
  \draw (a3) -- (b5);
  \draw (b1) -- (c1);
  \draw (b3) -- (c2);
\end{tikzpicture}
}
\caption{$T_A$
\label{figure.TA}
}
\end{subfigure}
\hspace{2cm}
\begin{subfigure}[b]{.25\linewidth}
\scalebox{0.7}{
\begin{tikzpicture}  
  [scale=.9,auto=center, every node/.style={circle,fill=blue!20}]
    
  \node (a1) at (0,0) {$-1$};  
  \node (a2) at (-1.5,-1.5)  {$0$};  
  \node (a3) at (1.5,-1.5) {$0$};  
  \node (b1) at (-2.5,-3) {$1$};
  \node (b2) at (-1.5,-3) {$1$};
  \node (b3) at (-0.5,-3) {$1$};
  \node (b4) at (0.5,-3) {$1$};
  \node (b5) at (2.5,-3) {$1$};
  \node (c1) at (-2.5,-4.5) {$2$};
  \node (c2) at (-0.5,-4.5) {$2$};
  
  \draw (a1) -- (a2); % these are the straight lines from one vertex to another  
  \draw (a1) -- (a3);  
  \draw (a2) -- (b1);
  \draw (a2) -- (b2);
  \draw (a2) -- (b3);
  \draw (a3) -- (b4);
  \draw (a3) -- (b5);
  \draw (b1) -- (c1);
  \draw (b3) -- (c2);
\end{tikzpicture}
}
\caption{$T_D$
\label{figure.TD}
}
\end{subfigure}
    \caption{Plane tree labellings $T_A$ and $T_D$ of the plane tree in Figure~\ref{figure.plane tree S}.}
\end{figure}
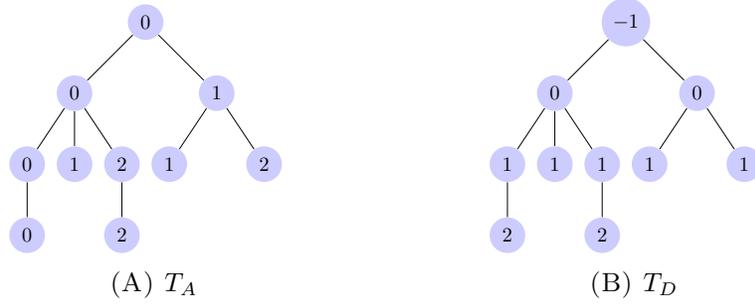

\begin{definition}
\label{definition.label D}
The \defn{labelling $D$} of a plane tree $T$, denoted by $T_{D}$, is defined by labelling a vertex $v$ by the number of edges in the path 
from $v$ to the root minus one.
\end{definition}

\begin{remark}
Note that the map $\lambda \colon \mathcal{P}_n \to \mathcal{L}_{n+1}$ of Definition~\ref{def: pf_tree_bijection} on parking functions
with car $i$ in row $i$ (or equivalently map $\eta$) sends the coinversions of labelled trees in the codomain to the labelling $D$ defined in Definition~\ref{definition.label D}.
\end{remark}

\begin{definition}
Let $T$ be a plane tree with $n+1$ vertices. The \defn{reading word of $T_{D}$}, denoted by $\readD(T)$, is defined by the following algorithm:
\renewcommand\labelenumi{(\arabic{enumi})}
\renewcommand\theenumi\labelenumi
	\begin{enumerate}
		\item Start by setting $\readD(T)$ to be an empty vector. Append the label of the root.
		\item If the length of $\readD(T)$ equals $n+1$, then remove the label corresponding to the root from $\readD(T)$ and output $\readD(T)$. 
		Otherwise consider the set of all vertices whose vertices have already been added to $\readD(T)$ but have at least one child whose label has 
		not been added. Find the vertex in this set with the largest label and call the vertex $v$. Attach to $\readD(T)$ the label of the leftmost child 
		of $v$ that has not already been added.
	\end{enumerate}
\end{definition}

This definition is also well-defined as vertices with the same labels will never be considered at the same time.

\begin{example}
The labelling $T_D$ of the tree $T$ in Figure~\ref{figure.plane tree S} is given in Figure~\ref{figure.TD}. The corresponding reading word is
$\readD(T) = (0,1,2,1,1,2,0,1,1)$.
\end{example}

\begin{definition}
Let $T$ be a plane tree. Let the $k$-th child of a vertex $v$ be the $k$-th leftmost child of $v$. We define the \defn{dual plane tree} of $T$, denoted by 
$T^{\mathsf{dual}}$, by the following algorithm:
\renewcommand\labelenumi{(\arabic{enumi})}
\renewcommand\theenumi\labelenumi
	\begin{enumerate}
		\item \textbf{Initialization:} Set $T^{\mathsf{dual}}$ to be a single vertex $u$ which we label as the root of $T^{\mathsf{dual}}$. If the root of $T$
		 has a child, then add a child to $u$ of $T^{\mathsf{dual}}$. Set this to be the $1$-st child of $u$ and associate this child with the $1$-st child of 
		 the root in $T$.
		\item \textbf{Determining if a non-root vertex $v$ in $T^{\mathsf{dual}}$ has a child:} Look at the associated vertex $v'$ of $v$ in the original 
		plane tree $T$. If $v'$ has a sibling to its right, then attach a child to $v$ which will be the $1$-st child of $v$. Associate the child of $v$ in 
		$T^{\mathsf{dual}}$ with the sibling directly right of $v'$ in $T$. If $v'$ has no sibling to its right, then $v$ has no children.
		\item \textbf{Determining if a vertex $v$ (including the root) in $T^{\mathsf{dual}}$ has a $k$-th child for $k >1$:} Let $w$ be the $(k-1)$-th 
		child of $v$. Look at the associated vertex $w'$ of $w$ in $T$. If $w'$ has a child, then attach a $k$-th child to $v$. Associate the $k$-th child of 
		$v$ to the $1$-st child of $w'$. If $w'$ has no children, then $v$ has no $k$-th child.
	\end{enumerate}
\end{definition}

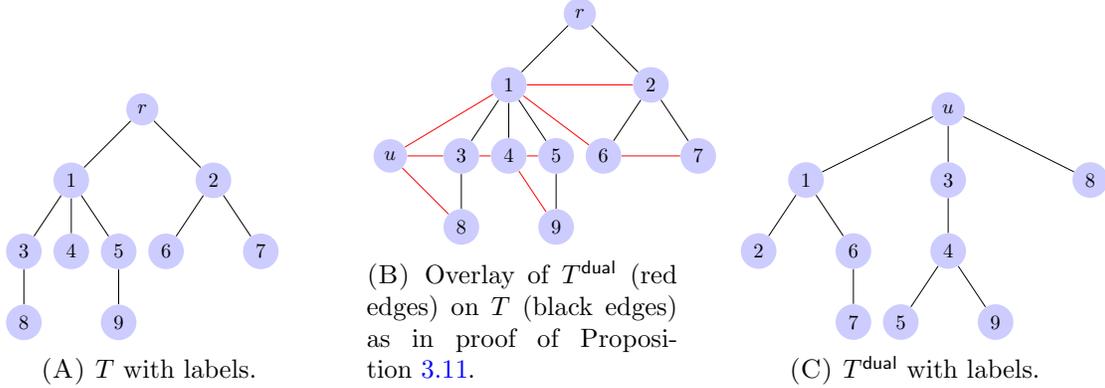
\begin{figure}[t]
\begin{subfigure}[b]{.25\linewidth}
\scalebox{0.7}{
\begin{tikzpicture}  
  [scale=.9,auto=center, every node/.style={circle,fill=blue!20}]
    
  \node (a1) at (0,0) {$r$};  
  \node (a2) at (-1.5,-1.5)  {$1$};  
  \node (a3) at (1.5,-1.5) {$2$};  
  \node (b1) at (-2.5,-3) {$3$};
  \node (b2) at (-1.5,-3) {$4$};
  \node (b3) at (-0.5,-3) {$5$};
  \node (b4) at (0.5,-3) {$6$};
  \node (b5) at (2.5,-3) {$7$};
  \node (c1) at (-2.5,-4.5) {$8$};
  \node (c2) at (-0.5,-4.5) {$9$};
  
  \draw (a1) -- (a2); % these are the straight lines from one vertex to another  
  \draw (a1) -- (a3);  
  \draw (a2) -- (b1);
  \draw (a2) -- (b2);
  \draw (a2) -- (b3);
  \draw (a3) -- (b4);
  \draw (a3) -- (b5);
  \draw (b1) -- (c1);
  \draw (b3) -- (c2);
\end{tikzpicture}
}
\caption{$T$ with labels.
\label{figure.T}
}
\end{subfigure}
\hspace{0.5cm}
\begin{subfigure}[b]{.25\linewidth}
\scalebox{0.7}{
\begin{tikzpicture}  
  [scale=.9,auto=center, every node/.style={circle,fill=blue!20}]
   
  \node (u) at (-4,-3) {$u$}; 
  \node (a1) at (0,0) {$r$};  
  \node (a2) at (-1.5,-1.5)  {$1$};  
  \node (a3) at (1.5,-1.5) {$2$};  
  \node (b1) at (-2.5,-3) {$3$};
  \node (b2) at (-1.5,-3) {$4$};
  \node (b3) at (-0.5,-3) {$5$};
  \node (b4) at (0.5,-3) {$6$};
  \node (b5) at (2.5,-3) {$7$};
  \node (c1) at (-2.5,-4.5) {$8$};
  \node (c2) at (-0.5,-4.5) {$9$};
  
  \draw (a1) -- (a2); % these are the straight lines from one vertex to another  
  \draw (a1) -- (a3);  
  \draw (a2) -- (b1);
  \draw (a2) -- (b2);
  \draw (a2) -- (b3);
  \draw (a3) -- (b4);
  \draw (a3) -- (b5);
  \draw (b1) -- (c1);
  \draw (b3) -- (c2);
  \draw[red] (u) -- (a2);
  \draw[red] (u) -- (b1);
  \draw[red] (u) -- (c1);
  \draw[red] (a2) -- (a3);
  \draw[red] (a2) -- (b4);
  \draw[red] (b1) -- (b2);
  \draw[red] (b2) -- (b3);
  \draw[red] (b2) -- (c2);
  \draw[red] (b4) -- (b5);
\end{tikzpicture}
}
\caption{Overlay of $T^{\mathsf{dual}}$ (red edges) on $T$ (black edges) as in proof of Proposition~\ref{prop: tree_involution}.
\label{figure.T overlay}
}
\end{subfigure}
\hspace{0.5cm}
\begin{subfigure}[b]{.3\linewidth}
\scalebox{0.7}{
\begin{tikzpicture}  
  [scale=.9,auto=center, every node/.style={circle,fill=blue!20}]
    
  \node (a1) at (0,0) {$u$};  
  \node (a2) at (-3,-1.5)  {$1$};  
  \node (a3) at (0,-1.5) {$3$};  
  \node (a4) at (3,-1.5) {$8$};
  \node (b1) at (-4,-3) {$2$};
  \node (b2) at (-2,-3) {$6$};
  \node (b3) at (0,-3) {$4$};
  \node (c1) at (-2,-4.5) {$7$};
  \node (c2) at (-1,-4.5) {$5$};
  \node (c3) at (1,-4.5) {$9$};
  
  \draw (a1) -- (a2); % these are the straight lines from one vertex to another  
  \draw (a1) -- (a3);  
  \draw (a1) -- (a4);
  \draw (a2) -- (b1);
  \draw (a2) -- (b2);
  \draw (a3) -- (b3);
  \draw (b2) -- (c1);
  \draw (b3) -- (c2);
  \draw (b3) -- (c3);
\end{tikzpicture}
}
\caption{$T^{\mathsf{dual}}$ with labels.
\label{figure.T dual}
}
\end{subfigure}
    \caption{Construction of the dual plane tree $T^{\mathsf{dual}}$ of the plane $T$ in Figure~\ref{figure.plane tree S}.}
\end{figure}

\begin{example}
The dual plane tree $T^{\mathsf{dual}}$ of the plane tree $T$ in Figure~\ref{figure.plane tree S} is given in Figure~\ref{figure.T dual}.
Observe, by comparing with Figure~\ref{figure.plane tree}, that in this example $T^{\mathsf{dual}} = \eta \circ \sigma^{-1}(T)$. This will be proved 
in general in Corollary~\ref{corollary.dual}.
\end{example}

It is easy to see that $T^{\mathsf{dual}} \in \mathcal{T}_{n+1}$ by observing that every non-root node of $T$ is paired with a non-root node of 
$T^{\mathsf{dual}}$, there are no loops in $T^{\mathsf{dual}}$, and the children of every vertex are given a proper ordering. To show that the term dual 
plane tree is not a misnomer, we also prove that this operation is an involution.

\begin{proposition} 
\label{prop: tree_involution}
Let $T$ be a plane tree. Then $(T^{\mathsf{dual}})^{\mathsf{dual}} = T$.
\end{proposition}

\begin{proof}
Draw the plane tree $T$ in the canonical way with every vertex sitting above all of its descendants and the order of its children increasing from left to 
right. Next place the root of $T^{\mathsf{dual}}$ to the left of all vertices in $T$ and draw the plane tree $T^{\mathsf{dual}}$ on top of $T$ such that 
any vertex in $T^{\mathsf{dual}}$ is drawn on top of its corresponding vertex in $T$. Under this configuration all vertices in $T^{\mathsf{dual}}$ sit to 
the left of their descendants, and the order of their children increase from top to bottom. Since a vertex $v$ and its corresponding vertex $v'$ lie on top 
of each other in the specified configuration, we will abuse notation and refer to both as vertex $v$. Interchanging the position of the two trees (i.e. flipping 
the plane along the perpendicular bisector of the two root nodes), we clearly see that for a vertex $v$ in $T$ its first child corresponds to the sibling on 
the right of $v$ in $T^{\mathsf{dual}}$ and its $k$-th child corresponds to the first sibling of the $(k-1)$-th child of $v$ for $k > 1$. Thus, 
$(T^{\mathsf{dual}})^{\mathsf{dual}} = T$.
\end{proof}

The two reading words are related under the dual map on plane trees.

\begin{proposition} 
\label{prop: dual_property}
Let T be a plane tree. Then
\[
	\readD(T^{\mathsf{dual}}) = \readA(T) \quad \text{and} \quad \readA(T^{\mathsf{dual}}) = \readD(T).
\]
\end{proposition}

\begin{proof}
It suffices to prove that $\readD(T) = \readA(T^{\mathsf{dual}})$ since this implies that 
$\readD(T^{\mathsf{dual}}) = \readA((T^{\mathsf{dual}})^{\mathsf{dual}})$ which equals $\readA(T)$ by Proposition~\ref{prop: tree_involution}. 

Let $\readD(T) = (r_{1}, r_{2}, \ldots, r_{n})$ and $\readA(T^{\mathsf{dual}}) = (s'_{1}, s'_{2}, \ldots, s'_{n})$. Let $v_{i}$ be the vertex in $T$ that 
has label $r_{i}$. Similarly, let $w_{i}$ be the vertex in $T^{\mathsf{dual}}$ that has label $s'_{i}$. We will prove by induction that 
$(r_{1}, r_{2}, \ldots, r_{k}) = (s'_{1}, s'_{2}, \ldots, s'_{k})$ and that $w_{k}$ corresponds to $v_{k}$ under $\mathsf{dual}$ for $1\leqslant k \leqslant n$. 
We have that both $w_{1}$ and $v_{1}$ are the leftmost child of their respective root nodes and the labelling of each is equal to zero. By the definition of 
$T^{\mathsf{dual}}$, we have $w_{1}$ corresponds to $v_{1}$. Assume that $(r_{1}, r_{2}, \ldots, r_{k}) = (s'_{1}, s'_{2}, \ldots, s'_{k})$ and that $w_{k}$ 
corresponds to $v_{k}$. If $v_{k+1}$ is a child of $v_{k}$, then $r_{k+1} = r_{k}+1$. Note that by definition of $T^{\mathsf{dual}}$, $w_{k}$ must have a sibling 
to its right. This implies that $w_{k+1}$ is the sibling directly right of $w_{k}$ and $w_{k+1}$ corresponds to $v_{k+1}$. We have 
$r_{k+1} = r_{k}+1 = s'_{k}+1 = s'_{k+1}$. 
If $v_{k+1}$ is not a child of $v_{k}$, then $v_{k+1}$ is the leftmost unvisited child of $y$, where $y=v_{i}$ for some $1\leqslant i < k$ and $y$ has the largest 
label out of all parents with unvisited children. Note as $v_{k}$ does not have any children, $w_{k}$ has no siblings to its right. Thus, to find $w_{k+1}$ 
we look for the leftmost child of the vertex $x$, where $x =w_{\ell}$ for some $1\leqslant \ell \leqslant k$ and $x$ has the largest label out of all parents with 
unvisited children. The condition that $x$ has unvisited children in $T^{\mathsf{dual}}$ implies that the parent of its corresponding vertex $x' = v_{\ell}$ in $T$ has 
an unvisited child. Thus the parent of $x'$ either is $y$ or has label smaller than $y$. If it has a label smaller than $y$ then by the definition of 
$T^{\mathsf{dual}}$ and our inductive hypothesis, there exists $v_{j}$ with $1 \leqslant 1 \leqslant k$ that has unvisited children and label strictly greater than 
$x$ which is a contradiction. Therefore $x'$ is the rightmost visited child of $y$ and the leftmost child of $x$ corresponds to the sibling to the right of $x'$. 
This implies that $w_{k+1}$ corresponds with $v_{k+1}$ and $w_{k+1} = w_{\ell} = v_{\ell} = v_{k+1}$.
\end{proof}

%%%%%%%%%%%%%%%%%%%%%%%%%%%%%%%%%%%%%%%%%%%%
\subsection{Symmetry of $F_n(q,t)$ and $G_n(q,t)$}
\label{section.symmetry}

In this section, we prove the symmetry of the polynomials $F_n(q,t)$ and $G_n(q,t)$. We do so by defining an involution on Dyck paths using the
Stanley and Haglund--Loehr maps $\sigma$ and $\eta$, which switches the area and depth statistics. We begin by relating the area and depth sequences
under the Stanley and Haglund--Loehr maps using the two reading words above. Recall that $a_i(\pi)$ and $d_i(\pi)$ are defined in 
Sections~\ref{section.dyck_paths} and~\ref{section.depth}.

\begin{proposition} \label{prop: sigma_property}
Let $\pi \in D_{n}$. Then
\[
\begin{split}
	\readD(\sigma(\pi)) &= (a_{1}(\pi), a_{2}(\pi), \ldots, a_{n}(\pi)),\\
	\readA(\sigma(\pi)) &= (d_{1}(\pi), d_{2}(\pi), \ldots, d_{n}(\pi)).
\end{split}
\]
\end{proposition}

\begin{proof}
Let $(r_{1}, r_{2}, \ldots, r_{n}) = \readD(\sigma(\pi)) $. We use induction on $1\leqslant k \leqslant n$ to prove that 
\[
	(r_{1}, r_{2}, \ldots, r_{k}) = (a_{1}(\pi), a_{2}(\pi), \ldots, a_{k}(\pi))
\]
and the $k$-th vertex (excluding the root) added in the creation of $\sigma(\pi)$ corresponds to the vertex with label $r_{k}$. Observe that 
$r_{1}$ corresponds to the label of the leftmost child of the root node. Note that this is the first node added in $\sigma(\pi)$. Thus, $r_{1} = 0 = a_{1}(\pi)$. 
Assume that  $(r_{1}, r_{2}, \ldots, r_{k}) = (a_{1}(\pi), a_{2}(\pi), \ldots, a_{k}(\pi))$ and $r_{k}$ is the label of the $k$-th vertex $v_{k}$ added in the 
creation of $\sigma(\pi)$ excluding the root. If the $(k+1)$-th vertex $v_{k+1}$ added to $\sigma(\pi)$ is a child of $v_{k}$, then in the Dyck path 
$a_{k+1}(\pi) = a_{k}(\pi)+1$. Since the label of $v_{k}$ was added last to $(r_{1}, \ldots, r_{k})$, we know that in the previous step the parent 
of $v_{k}$ had the largest label out of all parents containing a child whose label was not already appended to the reading word. As $v_{k}$ has a 
larger label than its parent and contains a child $v_{k+1}$, $r_{k+1}$ is the label of the leftmost available child of $v_{k}$ which would coincide with 
$v_{k+1}$. We have the label of $v_{k+1}$ is one more than $v_{k}$ giving us $r_{k+1} = r_{k}+1 = a_{k}(\pi)+1 = a_{k+1}(\pi)$. Now assume that 
$v_{k+1}$ is not a child of $v_{k}$. In the Dyck path, this corresponds to a block of East steps after the $k$-th North step. Let $\ell$ denote the size 
of this block of East steps. We see that $a_{k+1}(\pi) = a_{k}(\pi)+\ell-1$. In the tree, this corresponds to going $\ell$ vertices towards the root along 
the path from $v_{k}$ to the root and attaching a new vertex $v_{k+1}$ to this vertex $w$. Note that this implies that $v_{k}$ and all vertices strictly 
between $v_{k}$ and $w$ do not have any additional children that have not already been added. This implies that $w$ has the largest label 
of all vertices that contain a child whose label has not been appended to the reading word. Thus, $r_{k+1}$ corresponds to the label of $v_{k+1}$ 
which is one more than the label of $w$. Thus, $r_{k+1} = r_{k}- \ell +1 = a_{k}(\pi)- \ell +1 = a_{k+1}(\pi)$. By induction, we obtain
$\readD(\sigma(\pi)) = (a_{1}(\pi), a_{2}(\pi), \ldots, a_{n}(\pi))$.

Let $(s_{1}, s_{2}, \ldots, s_{n}) =  \readA(\sigma(\pi)) $. Similar to the previous paragraph, we use induction on $1\leqslant k \leqslant n$ to prove that
\[
	(s_{1}, s_{2}, \ldots, s_{k}) = (d_{1}(\pi), d_{2}(\pi), \ldots, d_{k}(\pi))
\]
and the North step corresponding to $d_{k}(\pi)$ created the vertex $v$ corresponding to the label $s_{k}$ in $\sigma(\pi)$. We have that $d_{1}(\pi) = 0$ 
corresponds to the first North step which created the leftmost child of the root node. Note that $s_{1} = 0$ and also corresponds to the leftmost vertex of 
the root node. Assume that $(s_{1}, s_{2}, \ldots, s_{k}) = (d_{1}(\pi), d_{2}(\pi), \ldots, d_{k}(\pi))$ and the North step corresponding to $d_{k}(\pi)$ in the 
Dyck path created the vertex $v_{k}$ corresponding to the label $s_{k}$ in $\sigma(\pi)$. If the vertex $v_{k+1}$ corresponding to $s_{k+1}$ is a sibling 
of $v_{k}$ then $s_{k+1} = s_{k}+1$. By the previous paragraph, siblings correspond to North steps on the same diagonal. Note that no other North step 
can lie between the diagonal connecting the North step $N_{k}$ of $v_{k}$ and the North step $N_{k+1}$ of $v_{k+1}$ (keep in mind that $N_{k}$ does not 
mean the $k$-th North step of $\pi$). Also, $N_{k+1}$ needs to be the bottommost North step in its column, otherwise $v_{k}$ and $v_{k+1}$ would not 
be siblings in $\sigma(\pi)$. Since the depth label $d_{k}(\pi)$ corresponds to $N_{k}$, we have that $d_{k+1}(\pi)$ is the labelling of $N_{k+1}$. Thus, 
$d_{k+1}(\pi) = d_{k}(\pi)+1 = s_{k}+1 = s_{k+1}$. Assume that the vertex $v_{k+1}$ corresponding to $s_{k+1}$ is not a sibling of $v_{k}$. This implies 
that $v_{k+1}$ is the leftmost child of the vertex $w$ with the largest labelling in $(s_{1}, s_{2}, \ldots, s_{k})$ whose children's labels have not been 
added yet. Looking at the North step $N_{k}$ corresponding to $d_{k}$, we have that the first North step reached by traveling northeast from $N_{k}$ 
is not in the bottom of its column. Thus to find the North step corresponding to $d_{k+1}(\pi)$, we must find the largest labeled cell visited by 
$(d_{1}(\pi), d_{2}(\pi), \ldots, d_{k}(\pi))$ that has a labelled cell directly above which has not been visited. Note that having a labeled cell directly 
above corresponds to having a child. Thus the North step corresponding to $d_{k+1}(\pi)$ is the same as the North step corresponding to 
$v_{k+1}$ and is one cell directly above the North step corresponding to $w$. Note that the labelling of $w$ is $s_{i}$ and the labelling of its 
corresponding North step is $d_{i}(\pi)$ for some $1\leqslant i \leqslant n$. As $v_{k+1}$ is the leftmost child of $w$, we have $s_{k+1}$ = $s_{i}$. 
Similarly, as $d_{k+1}(\pi)$ lies in the same column as $d_{i}(\pi)$, we have $d_{k+1}(\pi) = d_{i}(\pi)$. By induction $d_{i}(\pi) = s_{i}$, 
implying $d_{k+1}(\pi) = s_{k+1}$. By induction we obtain $\readA(\sigma(\pi)) = (d_{1}(\pi), d_{2}(\pi), \ldots, d_{n}(\pi))$.
\end{proof}

\begin{proposition} \label{prop: eta_property}
Let $\pi \in D_{n}$. Then
\[
\begin{split}
	\readA(\eta(\pi)) &= (a_{1}(\pi), a_{2}(\pi), \ldots, a_{n}(\pi)),\\
	\readD(\eta(\pi))& = (d_{1}(\pi), d_{2}(\pi), \ldots, d_{n}(\pi)).
\end{split}
\]
\end{proposition}

\begin{proof}
Let $x$ be the parking function obtained by labelling the North step in the $i$-th row by $i$. Then ~\cite{HL.2005} have showed the first equality. 

We prove the second equality by induction. Let $(r_{1}, r_{2}, \ldots, r_{n}) = \readD(\eta(\pi))$. We prove that 
$(r_{1}, r_{2}, \ldots, r_{k}) = (d_{1}(\pi), d_{2}(\pi), \ldots, d_{k}(\pi))$ for $1 \leqslant k \leqslant n$ and the North step corresponding to $d_{k}(\pi)$ 
created the vertex $v$ corresponding to the label $r_{k}$ in $\eta(\pi)$. We have that $d_{1}(\pi) = 0$ and it lies to the right of the first North step. 
The first North step under the map $\eta$ creates the leftmost child of the root which is precisely the vertex whose label is $r_{1} = 0$. Assume that  
$(r_{1}, r_{2}, \ldots, r_{k}) = (d_{1}(\pi), d_{2}(\pi), \ldots, d_{k}(\pi))$ and the North step corresponding to $d_{k}(\pi)$ created the vertex $v_{k}$ 
whose label is $r_{k}$. Let $v_{k+1}$ be the vertex whose label is $r_{k+1}$. Also define $N_{k}$ and $N_{k+1}$ to be the North steps that 
created $v_{k}$ and $v_{k+1}$, respectively. Assume that the vertex $v_{k+1}$ is a child of $v_{k}$. As $v_{k+1}$ is a child of $v_{k}$, we obtain
$r_{k+1}  = r_{k}+1$. By the definition of $\readD$, we have that $v_{k+1}$ is the leftmost child of $v_{k}$. This implies that their A label is the same.
Since $\readA(\eta(\pi)) = (a_{1}(\pi), a_{2}(\pi), \ldots, a_{n}(\pi))$, we have the North steps that created $v_{k}$ and $v_{k+1}$ under $\eta$ lie 
on the same diagonal. By the definition of $\eta$, we have that $N_{k+1}$ must be at the bottom of its column and no other North step lies between 
the $N_{k}$ and $N_{k+1}$. Thus $d_{k+1}(\pi)$ is the depth labelling of $N_{k+1}$ which satisfies $d_{k+1}(\pi) = d_{k}(\pi) +1 = r_{k}+1$. Assume 
now that $v_{k+1}$ is not a child of $v_{k}$ which implies by the definition of $\readD$ that $v_{k}$ does not have any children. Consider the subset 
$S'$ of $S = \{v_{1}, v_{2}, \dots, v_{k} \}$ containing all vertices with a child that is not also in $S$. Let $w$ be the vertex in $S'$ with the largest label. 
We have that $v_{k+1}$ is the leftmost child of $w$ that is not in $S$.  As $v_{k}$ does not have a child, the first North step attained by traveling Northeast 
from $N_{k}$ is not at the bottom of its column or does not exist. Thus to find $N_{k+1}$, we must find the largest labeled cell visited by 
$(d_{1}(\pi), d_{2}(\pi), \ldots, d_{k}(\pi))$ that has a labelled cell directly above which has not been visited. Note that having two North 
steps consecutively corresponds to them being siblings under $\eta$. Additionally, observe that the vertex in $S$ with the largest label 
out of vertices in $S$ containing a sibling not in $S$ is a child of $w$. Thus $v_{k+1}$ and the node created by $N_{k+1}$ are the same. All 
the children of $w$ have the same $D$ labelling, and depth labelings in the same column of $\pi$ are equal. Paired with the inductive hypothesis, 
this implies $r_{k+1} = d_{k+1}$.
\end{proof}

We are now ready to show that combining the Stanley and Haglund--Loehr maps gives an involution that interchanges area and depth.

\begin{proposition} \label{prop: main_bijection}
Let $\omega = \sigma ^{-1} \circ \eta \colon D_{n} \to D_{n}$. Then $\omega$ is an involution 
which interchanges the depth and area sequence.
\end{proposition}

\begin{proof}
By Propositions~\ref{prop: sigma_property} and ~\ref{prop: eta_property} we have that 
\[
\begin{split}
	(d_{1}(\omega(\pi)), d_{2}(\omega(\pi)), \ldots, d_{n}(\omega(\pi))) &= (a_{1}(\pi), a_{2}(\pi), \ldots, a_{n}(\pi)),\\
	(a_{1}(\omega(\pi)), a_{2}(\omega(\pi)), \ldots, a_{n}(\omega(\pi))) &= (d_{1}(\pi), d_{2}(\pi), \ldots, d_{n}(\pi)).
\end{split}
\]
Additionally, we have $(a_{1}(\omega^{2}(\pi)), a_{2}(\omega^{2}(\pi)), \ldots, a_{n}(\omega^{2}(\pi))) = (d_{1}(\omega(\pi)), d_{2}(\omega(\pi)), 
\ldots, d_{n}(\omega(\pi)))$ implying $(a_{1}(\omega^{2}(\pi)), a_{2}(\omega^{2}(\pi)), \ldots, a_{n}(\omega^{2}(\pi))) = (a_{1}(\pi), a_{2}(\pi), \ldots, a_{n}(\pi))$.
Since the area sequence uniquely determines a Dyck path, we have that $\omega$ is an involution. 
\end{proof}

\begin{figure}[t]
	\centering
	\begin{tikzpicture}[scale = .5]
	
	\draw[step=1.0, gray!60, thin] (0,0) grid (9,9);
	
	\draw[gray!60, thin] (0,0) -- (9,9);
	
	\draw[blue!60, line width=2pt] (0,0) -- (0,1) -- (0,2) -- (1,2) -- (1,3) -- (1,4) -- (2,4) -- (3,4) -- (4,4) -- (4,5) -- (4,6) -- (4,7) -- (5,7) -- (5,8) -- (6,8) -- (7,8) -- (8,8) -- (8,9) -- (9,9);
	
	\draw (0.5,0.5) node {$0$};
	\draw (0.5,1.5) node {$0$};
	\draw (1.5,2.5) node {$1$};
	\draw (1.5,3.5) node {$1$};
	\draw (4.5,4.5) node {$1$};
	\draw (4.5,5.5) node {$1$};
	\draw (4.5,6.5) node {$1$};
	\draw (5.5,7.5) node {$2$};
	\draw (8.5,8.5) node {$2$};

	\end{tikzpicture}
	\caption{$w(\pi)$ with $\pi$ as in Figure~\ref{figure.depth} with depth labelling.}
	\label{figure.dual}
\end{figure}
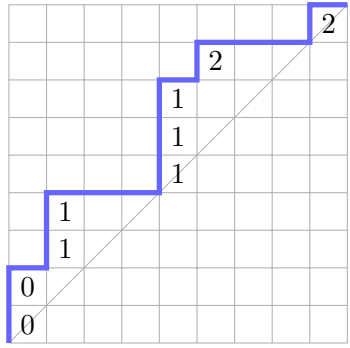

\begin{example}
Consider the Dyck path $\pi$ in Figure~\ref{figure.depth} with area and depth sequences (see also Example~\ref{example.depth})
\[
	a(\pi) = (0,1,2,1,1,2,0,1,1) \quad \text{and} \quad d(\pi) = (0,1,1,2,0,1,2,2,0).
\]
Then $\omega(\pi)$ is given in Figure~\ref{figure.dual} and it is easy to check that $a(\omega(\pi)) = d(\pi)$ and $d(\omega(\pi)) = a(\pi)$.
\end{example}

\begin{corollary}
 \label{cor: alt_main_bijection}
Let $\pi \in D_{n}$. Then $\omega(\pi) = \sigma^{-1}((\sigma(\pi))^{\mathsf{dual}}) = \eta^{-1}((\eta(\pi))^{\mathsf{dual}})$.
\end{corollary}

\begin{proof}
By Proposition~\ref{prop: main_bijection}, it suffices to prove that the area sequences of $\sigma^{-1}((\sigma(\pi))^{\mathsf{dual}})$ and 
$\eta^{-1}((\eta(\pi))^{\mathsf{dual}})$ are equal to the depth sequence of $\pi$. Using Propositions~\ref{prop: dual_property},~\ref{prop: sigma_property}, 
and~\ref{prop: eta_property}, we observe that this is indeed the case.
\end{proof}

\begin{corollary}
\label{corollary.dual}
Let $T\in \mathcal{T}_{n+1}$. Then $T^{\mathsf{dual}} = \eta \circ \sigma^{-1}(T)$.
\end{corollary}

\begin{proof}
This follows directly from Proposition~\ref{prop: main_bijection} and Corollary~\ref{cor: alt_main_bijection}.
\end{proof}

Finally, we are ready to prove the symmetry of $F_n(q,t)$ and $G_n(q,t)$.

\begin{theorem}
We have
\[
	F_{n}(q,t) = F_{n}(t, q) \qquad \text{and} \qquad G_{n}(q,t) = G_{n}(t, q).
\]
\end{theorem}

\begin{proof}
By Proposition~\ref{prop: main_bijection}, $\omega$ is a bijection on $D_{n}$ that interchanges the area and depth sequence of a Dyck path. 
As $\area$ and $\depth$ are defined as the sum of their respective sequences, we have that $\omega$ interchanges $\area$ and $\depth$, thereby
proving symmetry of $F_n(q,t)$.

By~\eqref{equation.dinv} and~\eqref{equation.ddinv}, the definitions of $\dinv$ and $\ddinv$ are identical except with the area and depth sequence
interchanged. Since by Proposition~\ref{prop: main_bijection} the involution $\omega$ interchanges the area and depth sequences, $\omega$ also 
interchanges $\dinv$ and $\ddinv$. Thus, $G_{n}(q,t)$ is symmetric in $q$ and $t$.
\end{proof}

From a similar argument, we obtain the following corollary.
\begin{corollary}
\label{corollary.Cqt depth}
We have
\[
	C_{n}(q,t) = \sum_{\pi \in D_{n}} q^{\depth(\pi)} t^{\ddinv(\pi)}.
\]
\end{corollary}

%%%%%%%%%%%%%%%%%%%%%%%%%%%%%%%%%%%%%%%%%%%%
\subsection{The Deutsch involution and $\omega$}
\label{section.deutsch}

We now define an involution $(\cdot)'$ on Dyck paths first introduced by Deutsch in~\cite{Deutsch.1999}.

\begin{definition}
\label{definition.deutsch}
We define $(\cdot)' \colon D_{n} \to D_{n}$ recursively as follows:
\renewcommand\labelenumi{(\arabic{enumi})}
\renewcommand\theenumi\labelenumi
\begin{enumerate}
	\item $\varepsilon' = \varepsilon$, where $\varepsilon$ is the empty Dyck path.
	\item For $\pi \in D_{n}$ and $n\geqslant 1$, write $\pi = N \alpha E \beta$, where $\alpha$ and $\beta$ are Dyck paths.
	Note that $\alpha, \beta$ are allowed to be empty. Then define $\pi ' = N \beta' E \alpha'$.
\end{enumerate}
\end{definition}

The map $\omega = \sigma^{-1} \circ \eta$ gives an explicit description of Deutsch's recursive operator as we first observed
using FindStat~\cite{FindStat}.

\begin{proposition}
\label{proposition.deutsch}
Let $\pi \in D_{n}$. Then $\omega(\pi) = \pi'$.
\end{proposition}

\begin{proof}
By Proposition~\ref{prop: main_bijection}, it suffices to prove that
\[
	(d_{1}(\pi), d_{2}(\pi), \ldots, d_{n}(\pi)) = (a_{1}(\pi'), a_{2}(\pi'), \ldots, a_{n}(\pi')).
\]
We proceed by induction on $n$. We have that both the area and depth sequence of $\varepsilon$ are $\emptyset$ . Assume that 
$(d_{1}(\pi), d_{2}(\pi), \ldots, d_{j}(\pi)) = (a_{1}(\pi'), a_{2}(\pi'), \ldots, a_{j}(\pi'))$ for all $\pi \in D_{j}$, where $0 \leqslant j \leqslant n$. Let $\pi \in D_{n+1}$ 
and let $\alpha$ and $\beta$ be Dyck paths such that $\pi = N\alpha E \beta$. Let $k-1$ be the semilength of $\alpha$. We have 
that $(k,k)$ is the first time the path $\pi$ touches the diagonal after $(0,0)$. From the definition of the depth labelling and the argument in the proof of 
Proposition~\ref{prop: recurrence}, we have $(d_{1}(\pi), d_{2}(\pi), \ldots, d_{n+1}(\pi)) = (0, d_{1}(\beta)+1, d_{2}(\beta), \ldots, d_{n+1-k}(\beta), 
d_{1}(\alpha), d_{2}(\alpha), \ldots, d_{k-1}(\alpha))$. From the definition of the area sequence and $(\cdot)'$, we have that 
$(a_{1}(\pi'), a_{2}(\pi'), \ldots, a_{n+1}(\pi')) = (0, a_{1}(\beta')+1, a_{2}(\beta')+1, \ldots, a_{n+1-k}(\beta')+1, a_{1}(\alpha'), a_{2}(\alpha'),  
\ldots, a_{k-1}(\alpha'))$. Note that $\alpha$ and $\beta$ have semilength strictly less than $n+1$. Hence by induction 
$(d_{1}(\beta), \ldots, d_{n+1-k}(\beta)) = (a_{1}(\beta'),\ldots, a_{n+1-k}(\beta'))$ and $(d_{1}(\alpha), d_{2}(\alpha), \ldots, d_{k-1}(\alpha)) 
= (a_{1}(\alpha'), a_{2}(\alpha'),  \ldots, a_{k-1}(\alpha'))$. Thus, $(d_{1}(\pi), d_{2}(\pi), \ldots, d_{n+1}(\pi)) = (a_{1}(\pi'), a_{2}(\pi'), \ldots, a_{n+1}(\pi'))$.
\end{proof}

Using Corollary~\ref{cor: alt_main_bijection} and Proposition~\ref{proposition.deutsch}, we find a relation between the $(\cdot)^{\mathsf{dual}}$ operator 
defined on plane trees and the one defined on Dyck paths.

\begin{corollary}
The following diagram commutes:
\[\begin{tikzcd}
D_{n} \arrow[leftrightarrow]{r}{(\cdot)'} \arrow[swap]{d}{\sigma \text{ or }\eta} & D_{n} \arrow{d}{\sigma \text{ or }\eta} \\
\mathcal{T}_{n+1} \arrow[leftrightarrow]{r}{(\cdot)^{\mathsf{dual}}} & \mathcal{T}_{n+1}.
\end{tikzcd}
\]
\end{corollary}

Deutsch proved~\cite{Deutsch.1999} that the operator $(\cdot)'$ interchanges the initial rise ($\mathsf{IR}$) of a Dyck path (the number of North steps 
before the first East step) with its number of returns ($\mathsf{RET}$) (the number of times the Dyck path touches the diagonal excluding the point $(0,0)$). 
We see that the initial rise and the number of returns of a Dyck path correspond to the length of the leftmost path from the root to a leaf and the number 
of children of the root, respectively, under $\sigma$ (and vice versa under $\eta$). This gives an alternate explanation of the symmetry of the 
\defn{Tutte polynomial} 
\begin{equation*}
	T_{\mathsf{Cat}_{n}}(q, t) = \sum_{\pi \in D_{n}} q^{\mathsf{IR}(\pi)}t^{\mathsf{RET}(\pi)}
\end{equation*}
associated with the Catalan matroid $\mathsf{Cat}_{n}$ defined in~\cite{Ardila.2003}.

Stump~\cite{Stump.2014} proved that the coefficient of $q^{a}t^{b}$ of $T_{\mathsf{Cat}_{n}}(q, t)$ only depends on the sum $a+b$ using a map 
given by Speyer~\cite{Speyer.2013}. This map $\tau$ fixes Dyck paths $\pi$, where $\mathsf{RET}(\pi) = 1$ and sends Dyck paths 
$\pi = N\alpha_{1}EN\alpha_{2}EN\alpha_{3}E\ldots N \alpha_{k}E$ to $NN\alpha_{1}E\alpha_{2}EN\alpha_{3}E\ldots N \alpha_{k}E$, where 
$\mathsf{RET}(\pi) = k >1$ and $\alpha_{i}$ is a Dyck path that is possibly empty. Speyer's map has a nice relation with $\omega$ as follows.

\begin{proposition}
Let $\pi \in D_{n}$. Then $\tau^{-1}\circ\omega(\pi) = \omega \circ \tau(\pi)$.
\end{proposition}

\begin{proof}
If $\mathsf{RET}(\pi) = 1$, then $\tau(\pi) = \pi$ and $\omega \circ \tau(\pi) = \omega(\pi)$ . As $\omega$ interchanges initial rises and the number 
of returns, we have $\mathsf{IR}(\omega(\pi)) = 1$. This implies that $\tau^{-1}\circ\omega(\pi) = \omega(\pi)$. Thus, we have 
$\tau^{-1}\circ\omega(\pi) = \omega \circ \tau(\pi)$.

If $\mathsf{RET}(\pi) = k>1$, let $\pi = N\alpha_{1}EN\alpha_{2}EN\alpha_{3}E\ldots N \alpha_{k}E$, where $\alpha_{i}$ is a possibly empty 
Dyck path. We show $\omega(\pi) = \tau\circ\omega\circ\tau(\pi)$. From Definition~\ref{definition.deutsch} and Proposition~\ref{proposition.deutsch}
\[
	\omega(\pi) = N(N\alpha_{2}EN\alpha_{3}E\ldots N \alpha_{k}E)'E\alpha_{1}'.
\] 
On the other hand,
\begin{align*}
	\tau(\pi) &= NN\alpha_{1}E\alpha_{2}EN\alpha_{3}E\ldots N \alpha_{k}E,\\
	\omega\circ\tau(\pi) &= N(N\alpha_{3}E\ldots N \alpha_{k}E)'E(N\alpha_{1}E\alpha_{2})'\\
	&=  N(N\alpha_{3}E\ldots N \alpha_{k}E)'EN\alpha_{2}'E\alpha_{1}',\\
	\tau\circ\omega\circ\tau(\pi) &=  NN(N\alpha_{3}E\ldots N \alpha_{k}E)'E\alpha_{2}'E\alpha_{1}'\\
	&= N(N\alpha_{2}EN\alpha_{3}E\ldots N \alpha_{k}E)'E\alpha_{1}'.
\end{align*}
Hence, $\omega(\pi) = \tau\circ\omega\circ\tau(\pi)$.
\end{proof}

%%%%%%%%%%%%%%%%%%%%%%%%%%%%%%%%%%%%%%%%%%%%
\subsection{Parking Functions}
\label{section.parking_functions}

Kreweras~\cite{Kreweras.1980} essentially proved recursively
\begin{equation}
\label{equation.kreweras}
	\sum_{T \in \mathcal{L}_{n+1}} q^{\coinv(T)} = \sum_{\pi \in \mathcal{P}_{n}} q^{\area(\pi)}. 
\end{equation}
Combining this with~\eqref{equation: coinv_graph}, one obtains the formula
\begin{equation}
q^{n} \sum_{\pi \in \mathcal{P}_{n}} (1+q)^{\area(\pi)}= \sum _{G \in \mathcal{C}_{n+1}} q^{e(G)},
\end{equation}
which was also observed in~\cite{AP.2018}.
Following in Gessel and Wang's footsteps~\cite{GesselWang.1979}, we provide a combinatorial proof of this formula.

We start by defining an algorithm that produces a specific spanning tree from a labelled connected graph. Recall from 
Section~\ref{section.trees and parking_functions} that $\mathcal{L}_n$ is the set of all labelled trees on $n$ vertices and
$\mathcal{C}_n$ is the set of all labelled connected graphs with vertex set $\{0,\ldots,n-1\}$.

\begin{definition} 
Let $\mathcal{S}\colon \mathcal{C}_{n} \to \mathcal{L}_{n}$ be given by the following algorithm:
\renewcommand\labelenumi{(\arabic{enumi})}
\renewcommand\theenumi\labelenumi
\begin{enumerate} 
	\item Start with all vertices of $G \in \mathcal{C}_{n}$ in the ``Not Seen" state.
	\item Visit vertex $0$ and set its state to ``Seen". Visit all vertices $v$ adjacent to $0$ in increasing label order including the edges from vertex $0$ to $v$.
	\item If all vertices of $G$ are in the ``Seen" state, then return the subgraph of $G$ comprised of all vertices and edges that were visited. 
	Otherwise, find the vertex $v$ that was visited last and is in the ``Not Seen" state. Set $v$ to ``Seen''. Visit all ``Not Seen" vertices $w$ adjacent 
	to $v$ that have not been visited already including the edge between $v$ and $w$, where vertices with smaller labels are visited first.
\end{enumerate}
\end{definition}

Clearly, $\mathcal{S}(G)$ is connected and acyclic for any $G \in \mathcal{C}_{n}$ implying that $\mathcal{S}$ is well defined. For $T \in \mathcal{L}_{n}$, 
let $\mathcal{G}_{\mathcal{S}}(T)$ denote the set of all labelled connected graphs $G$ satisfying $\mathcal{S}(G) = T$.

\begin{example}
Consider the labelled connected graph $G$ in Figure~\ref{figure.G}. Its spanning tree $\mathcal{S}(G)$ is given in Figure~\ref{figure.S(G)}.
\end{example}

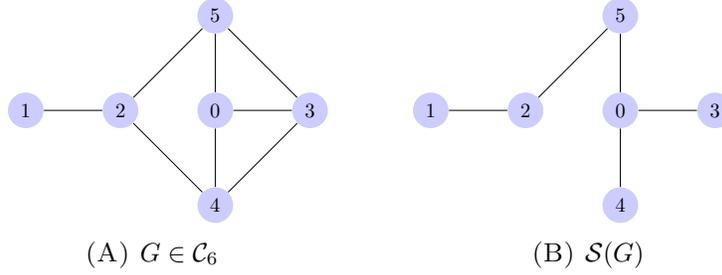
\begin{figure}[t]
\begin{subfigure}[b]{.25\linewidth}
\scalebox{0.7}{
\begin{tikzpicture}  
  [scale=.9,auto=center, every node/.style={circle,fill=blue!20}]
    
  \node (a1) at (0,0) {$1$};  
  \node (a2) at (2,0)  {$2$};  
  \node (a0) at (4,0) {$0$};  
  \node (a3) at (6,0) {$3$};
  \node (a5) at (4,2) {$5$};
  \node (a4) at (4,-2) {$4$};
  
  \draw (a1) -- (a2); % these are the straight lines from one vertex to another  
  \draw (a2) -- (a5);  
  \draw (a0) -- (a5);
  \draw (a3) -- (a5);
  \draw (a4) -- (a0);
  \draw (a4) -- (a3);
  \draw (a2) -- (a4);
  \draw (a0) -- (a3);
\end{tikzpicture}
}
\caption{$G \in \mathcal{C}_6$
\label{figure.G}
}
\end{subfigure}
\hspace{1cm}
\begin{subfigure}[b]{.3\linewidth}
\scalebox{0.7}{
\begin{tikzpicture}  
  [scale=.9,auto=center, every node/.style={circle,fill=blue!20}]
    
  \node (a1) at (0,0) {$1$};  
  \node (a2) at (2,0)  {$2$};  
  \node (a0) at (4,0) {$0$};  
  \node (a3) at (6,0) {$3$};
  \node (a5) at (4,2) {$5$};
  \node (a4) at (4,-2) {$4$};
  
  \draw (a1) -- (a2); % these are the straight lines from one vertex to another  
  \draw (a2) -- (a5);  
  \draw (a0) -- (a5);
  \draw (a4) -- (a0);
  \draw (a0) -- (a3);
\end{tikzpicture}
}
\caption{$\mathcal{S}(G)$
\label{figure.S(G)}
}
\end{subfigure}
    \caption{Connected labelled graph $G$ and its spanning tree $\mathcal{S}(G)$.
    \label{figure.connected graph}}
\end{figure}

We will now associate a set of labelled connected graphs to a labelled tree by adding certain edges to the tree.  Let $T \in \mathcal{L}_{n}$. To each vertex 
$i$ in $T$, we associate a set of edges $\mathcal{E}_{T}(i)$ that are not in $T$ as follows. Let $Q$ be the unique path from $i$ to the root node $0$ in
$T$. We let
\[
	\mathcal{E}_{T}(i) = \{ \{i, j \} \mid j \text{ is a sibling of some vertex } k \in Q \text{ and } j < k \}.
\]
Define $\mathcal{E}_{T} = \bigcup_{i =0}^{n} \mathcal{E}_{T}(i)$ and let $\mathcal{G}_{\mathcal{E}}(T)$ denote the set of all connected graphs obtained by 
adding some subset of edges from $\mathcal{E}_{T}$ to $T$.

\begin{example}
Let $T$ be the labelled tree in Figure~\ref{figure.S(G)}. Then
\[
	\mathcal{E}_T = \{\{1,3\}, \{1,4\}, \{2,3\}, \{2,4\}, \{5,3\}, \{5,4\}, \{4,3\}\}.
\]
\end{example}

\begin{proposition}
\label{prop: connected_graph_bijection}
Let $T \in \mathcal{L}_{n}$. Then $\mathcal{G}_{\mathcal{S}}(T) = \mathcal{G}_{\mathcal{E}}(T)$.
\end{proposition}

\begin{proof}
Let $G \in \mathcal{G}_{\mathcal{S}}(T)$. This implies that $G = T \sqcup S$, where $S$ is a set of edges not in $T$. Assume 
$S \not \subseteq \mathcal{E}_{T}$ and let $e = \{v, w\} \in S - \mathcal{E}_{T}$, where $v$ was ``Seen'' before $w$ in the construction of 
$\mathcal{S}(G)$. Observe that in the step, where $v$ is marked as ``Seen", all visited vertices that are ``Not Seen" are endpoints of edges in
$\mathcal{E}_{T}$. Hence, vertex $w$ has not been visited when $v$ was marked as ``Seen". This implies that $v$ is a parent of $w$ in $T$ which 
contradicts $e$ being an edge not in $T$. Therefore, $\mathcal{G}_{\mathcal{S}}(T) \subseteq \mathcal{G}_{\mathcal{E}}(T)$.

Let $G = T \sqcup S$, where $S \subseteq \mathcal{E}_{T}$ and assume $\mathcal{S}(G) \not= T$. Let edge $e = \{v, w\} \in S$ be the first edge
used in $\mathcal{S}(G)$ that is not present in $T$. Assume $v$ is marked ``Seen" before vertex $w$ in the construction of $\mathcal{S}(G)$. Recall 
that $w$ is a smaller sibling of a vertex on the path from $v$ to the root in $T$. This implies that $w$ has already been visited when $v$ is marked 
as ``Seen", and thus the edge $e$ cannot have been used in the construction of $\mathcal{S}(G)$. Therefore, 
$\mathcal{G}_{\mathcal{E}}(T) \subseteq \mathcal{G}_{\mathcal{S}}(T)$ and $\mathcal{G}_{\mathcal{S}}(T) = \mathcal{G}_{\mathcal{E}}(T)$.
\end{proof}

Observe that the following relation holds between the number of associated edges of a vertex to the statistic defined after Definition~\ref{def: pf_tree_bijection}.

\begin{lemma} \label{lemma: area_edges}
Let $i$ be a vertex in $T \in \mathcal{L}_n$. Then $\lvert \mathcal{E}_{T}(i) \rvert = \tilde{d}_{i}(T)$.
\end{lemma}

\begin{proof}
We induct on the distance of vertex $i$ to the root, where distance is defined as the length of the path between the two vertices. The only vertex 
that is distance $0$ from the root is the root itself. We clearly have $\mathcal{E}_{T}(0) = \emptyset$ and $\tilde{d}_{0}(T) = 0$. Assume that 
$\lvert \mathcal{E}_{T}(i) \rvert = \tilde{d}_{i}(T)$ for all vertices $i$ that are distance $m$ from the root. Let $j$ be a vertex that is distance $m+1$ 
from the root and let $Q$ be the unique path from $j$ to the root. By assumption, $\lvert \mathcal{E}_{T}(i) \rvert = \tilde{d}_{i}(T)$, where $i$ is the 
parent of $j$. Let $S$ be the set of all siblings of $j$ that are smaller than $j$. Observe that $\mathcal{E}_{T}(j) = \mathcal{E}_{T}(i) \cup S$ 
and $\tilde{d}_{j}(T) = \tilde{d}_{i}(T) + k-1$ where $j$ is the $k$-th smallest child of $i$. Thus, $\lvert \mathcal{E}_{T}(j) \rvert = \tilde{d}_{j}(T)$.
\end{proof}

We now prove~\eqref{equation.kreweras} combinatorially.

\begin{theorem} 
\label{thm: pf_graph_formula}
We have 
\[
	q^{n} \sum_{\pi \in \mathcal{P}_{n}} (1+q)^{\area(\pi)}= \sum _{G \in \mathcal{C}_{n+1}} q^{e(G)}.
\]
\end{theorem}

\begin{proof}
By Proposition~\ref{prop: connected_graph_bijection} and using the fact that a tree on $n+1$ vertices has $n$ edges, we observe
\begin{equation}
	 \sum _{G \in \mathcal{C}_{n+1}} q^{e(G)}=  \sum_{T \in \mathcal{L}_{n+1}} q^{e(T)}(1+q)^{\lvert\mathcal{E}_{T}\rvert}
	 = \sum_{T \in \mathcal{L}_{n+1}} q^{n}(1+q)^{\lvert\mathcal{E}_{T}\rvert}.
\end{equation}
Using Lemma~\ref{lemma: area_edges}, Definition~\ref{def: pf_tree_bijection}, and~\eqref{equation: area_tree}, we have
\begin{equation}
	\sum_{T \in \mathcal{L}_{n+1}} q^{n}(1+q)^{\lvert\mathcal{E}_{T}\rvert} = \sum_{T \in \mathcal{L}_{n+1}} q^{n}(1+q)^{\sum_{i=0}^{n} \tilde{d}_{i}(T)} 
	=  \sum_{\pi \in \mathcal{P}_{n}} q^{n}(1+q)^{\area(\pi)}.
\end{equation}
Combining the equations above, we obtain the desired result.
\end{proof}

Substituting $q = 1$ into Theorem~\ref{thm: pf_graph_formula} gives the following result which provides an explicit proof of a remark found 
in~\cite[Section 3]{GHQR.2019}.

\begin{corollary} 
The following identity holds
\[
	\sum_{\pi \in \mathcal{P}_{n}} 2^{\area(\pi)} = \lvert \mathcal{C}_{n+1} \rvert.
\]
\end{corollary}

From Theorem~\ref{thm: pf_graph_formula} and~\eqref{equation: coinv_graph}, we obtain a new proof of the fact that $\mathsf{area}$ and $\mathsf{coinv}$ 
are equidistributed over labelled trees/parking functions~\cite{IrvingRattan.2021}. 

\begin{corollary}
The following identity holds
\[
	\sum_{T \in \mathcal{L}_{n}} q^{\area(\lambda^{-1}(T))} = \sum_{T \in \mathcal{L}_{n}} q^{\coinv(T)}.
\]
\end{corollary}

%%%%%%%%%%%%%%%%%%%%%%%%%%%%%%%%%%%%%%%%%%%%%%%%%%%
\bibliographystyle{alpha}
\bibliography{catalan}{}

\newcommand{\etalchar}[1]{$^{#1}$}
\begin{thebibliography}{GHQR19}

\bibitem[AKOP02]{AKOP.2002}
George~E. Andrews, Christian Krattenthaler, Luigi Orsina, and Paolo Papi.
\newblock ad-nilpotent {$\mathfrak{b}$}-ideals in {${\rm sl}(n)$} having a
  fixed class of nilpotence: combinatorics and enumeration.
\newblock {\em Trans. Amer. Math. Soc.}, 354(10):3835--3853, 2002.

\bibitem[ALW15]{ALW.2015}
Drew Armstrong, Nicholas~A. Loehr, and Gregory~S. Warrington.
\newblock Sweep maps: a continuous family of sorting algorithms.
\newblock {\em Adv. Math.}, 284:159--185, 2015.

\bibitem[AP18]{AP.2018}
Per Alexandersson and Greta Panova.
\newblock L{LT} polynomials, chromatic quasisymmetric functions and graphs with
  cycles.
\newblock {\em Discrete Math.}, 341(12):3453--3482, 2018.

\bibitem[AR14]{AdinRoichman.2014}
Ron~M. Adin and Yuval Roichman.
\newblock On maximal chains in the non-crossing partition lattice.
\newblock {\em J. Combin. Theory Ser. A}, 125:18--46, 2014.

\bibitem[Ard03]{Ardila.2003}
Federico Ardila.
\newblock The {C}atalan matroid.
\newblock {\em J. Combin. Theory Ser. A}, 104(1):49--62, 2003.

\bibitem[BM96]{BM.1996}
S.~Benchekroun and P.~Moszkowski.
\newblock A new bijection between ordered trees and legal bracketings.
\newblock {\em European J. Combin.}, 17(7):605--611, 1996.

\bibitem[CDH16]{CDH.2016}
Cesar Ceballos, Tom Denton, and Christopher R.~H. Hanusa.
\newblock Combinatorics of the zeta map on rational {D}yck paths.
\newblock {\em J. Combin. Theory Ser. A}, 141:33--77, 2016.

\bibitem[CFM20]{CFM.2020}
Cesar Ceballos, Wenjie Fang, and Henri M\"{u}hle.
\newblock The steep-bounce zeta map in parabolic {C}ataland.
\newblock {\em J. Combin. Theory Ser. A}, 172:105210, 59, 2020.

\bibitem[Deu99]{Deutsch.1999}
Emeric Deutsch.
\newblock An involution on {D}yck paths and its consequences.
\newblock {\em Discrete Math.}, 204(1-3):163--166, 1999.

\bibitem[GH96]{GarsiaHaiman.1996}
A.~M. Garsia and M.~Haiman.
\newblock A remarkable {$q,t$}-{C}atalan sequence and {$q$}-{L}agrange
  inversion.
\newblock {\em J. Algebraic Combin.}, 5(3):191--244, 1996.

\bibitem[GH02]{GarsiaHaglund.2002}
A.~M. Garsia and J.~Haglund.
\newblock A proof of the {$q,t$}-{C}atalan positivity conjecture.
\newblock volume 256, pages 677--717. 2002.
\newblock LaCIM 2000 Conference on Combinatorics, Computer Science and
  Applications (Montreal, QC).

\bibitem[GHQR19]{GHQR.2019}
Adriano~M Garsia, James Haglund, Dun Qiu, and Marino Romero.
\newblock $e$-{P}ositivity results and conjectures.
\newblock preprint \arxiv{1904.07912}, 2019.

\bibitem[GW79]{GesselWang.1979}
Ira Gessel and Da~Lun Wang.
\newblock Depth-first search as a combinatorial correspondence.
\newblock {\em J. Combin. Theory Ser. A}, 26(3):308--313, 1979.

\bibitem[Hag03]{Haglund.2003}
J.~Haglund.
\newblock Conjectured statistics for the {$q,t$}-{C}atalan numbers.
\newblock {\em Adv. Math.}, 175(2):319--334, 2003.

\bibitem[Hag08]{Haglund.2008}
James Haglund.
\newblock {\em The {$q$},{$t$}-{C}atalan numbers and the space of diagonal
  harmonics}, volume~41 of {\em University Lecture Series}.
\newblock American Mathematical Society, Providence, RI, 2008.
\newblock With an appendix on the combinatorics of Macdonald polynomials.

\bibitem[HL05]{HL.2005}
J.~Haglund and N.~Loehr.
\newblock A conjectured combinatorial formula for the {H}ilbert series for
  diagonal harmonics.
\newblock {\em Discrete Math.}, 298(1-3):189--204, 2005.

\bibitem[Inc21]{Sloane}
OEIS~Foundation Inc.
\newblock The {O}n-{L}ine {E}ncyclopedia of {I}nteger {S}equences, {A}116395.
\newblock \url{https://oeis.org/A116395}, 2021.

\bibitem[IR21]{IrvingRattan.2021}
John Irving and Amarpreet Rattan.
\newblock Trees, parking functions and factorizations of full cycles.
\newblock {\em European J. Combin.}, 93:Paper No. 103257, 22, 2021.

\bibitem[Kre80]{Kreweras.1980}
G.~Kreweras.
\newblock Une famille de polyn\^{o}mes ayant plusieurs propri\'{e}t\'{e}s
  \'{e}numeratives.
\newblock {\em Period. Math. Hungar.}, 11(4):309--320, 1980.

\bibitem[RS{\etalchar{+}}]{FindStat}
Martin Rubey, Christian Stump, et~al.
\newblock {FindStat} - {T}he combinatorial statistics database.
\newblock \url{http://www.FindStat.org}.
\newblock Accessed: \today.

\bibitem[Spe13]{Speyer.2013}
David Speyer.
\newblock A double grading of {C}atalan numbers.
\newblock \url{http://www.mathoverflow.net/questions/131809}, 2013.

\bibitem[Sta15]{Stanley.2015}
Richard~P. Stanley.
\newblock {\em Catalan numbers}.
\newblock Cambridge University Press, New York, 2015.

\bibitem[Stu14]{Stump.2014}
Christian Stump.
\newblock On a new collection of words in the {C}atalan family.
\newblock {\em J. Integer Seq.}, 17(7):Article 14.7.1, 8, 2014.

\bibitem[TW18]{TW.2018}
Hugh Thomas and Nathan Williams.
\newblock Sweeping up zeta.
\newblock {\em Selecta Math. (N.S.)}, 24(3):2003--2034, 2018.

\end{thebibliography}

\end{document}